\documentclass[10pt]{amsart}
\usepackage{amsmath,amsfonts,amssymb,amsthm}
\usepackage{graphics,graphicx}
\usepackage[colorlinks,hyperindex,linkcolor=blue]{hyperref}
\hypersetup{
pdfauthor = {L. Dahinden, A. del Pino Gomez},
 pdftitle= {sub-Riemannian Billiards},
 pdfsubject = {sub-Riemannian Geometry, billiards, tangent distributions, geodesics}}
\usepackage{tikz-cd}
\usepackage{comment}

\newcommand{\gfrac}{\mathfrak{g}}
\newcommand{\hfrac}{\mathfrak{h}}

\newcommand{\Acal}{{\mathcal{A}}}

\newcommand{\Ccal}{{\mathcal{C}}}

\newcommand{\Fcal}{{\mathcal{F}}}

\newcommand{\Lcal}{{\mathcal{L}}}

\newcommand{\DD}{\mathbb{D}}
\newcommand{\EE}{\mathbb{E}}

\newcommand{\NN}{\mathbb{N}}

\newcommand{\RR}{\mathbb{R}}
\renewcommand{\SS}{\mathbb{S}}

\newcommand{\ZZ}{\mathbb{Z}}

\newcommand{\sym}{\operatorname{Sym}}

\newcommand{\Ann}{\operatorname{Ann}}

\newcommand{\SO}{\operatorname{SO}}

\newcommand{\area}{\operatorname{Area}}
\newcommand{\cotan}{\operatorname{cotan}}
\newcommand{\rank}{\operatorname{rank}}

\newcommand{\Endpoint}{{\mathfrak{E}\operatorname{p}}}
\newcommand{\Op}{{\mathcal{O}p}}
\newcommand{\len}{{\operatorname{len}}}
\newcommand{\std}{{\operatorname{std}}}
\newcommand{\Heis}{{\operatorname{Heis}}}
\newcommand{\wavefront}{{\operatorname{WF}}}
\newcommand{\sgn}{{\operatorname{sgn}}}
\newcommand{\Vertical}{{\operatorname{Vert}}}

\newtheorem*{theorem*}{Theorem}

\newtheorem{proposition}{Proposition}[section]
\newtheorem{lemma}[proposition]{Lemma}

\newtheorem{corollary}[proposition]{Corollary}
\newtheorem{definition}[proposition]{Definition}

\theoremstyle{remark}
\newtheorem{remark}[proposition]{Remark}
\newtheorem{question}[proposition]{Question}
\newtheorem{example}[proposition]{Example}

\newtheorem{assumption}[proposition]{Assumption}

\numberwithin{equation}{section}

\begin{document}

\title{Introducing Sub-Riemannian and sub-Finsler Billiards}
\date{\today}

\thanks{The second author would like to thank Aaron Gootjes-Dreesbach for providing comments on the first version of the paper. The first author is supported by Deutsche Forschungsgemeinschaft (DFG) under Germany's Excellence Strategy EXC-2181/1 - 390900948 (the Heidelberg STRUCTURES Excellence Cluster), as well as by SFB/TRR 191 ``Symplectic Structures in Geometry, Algebra and Dynamics'' funded by the DFG. The second author is supported by the NWO 016.Veni.192.013 grant.}


\author{Lucas Dahinden}
\address{Lucas Dahinden, Universität Heidelberg, Mathematisches Institut, Mathematikon, Im Neuenheimer Feld 205, 69120 Heidelberg, Germany}
\email{ldahinden@mathi.uni-heidelberg.de}

\author{\'Alvaro del Pino}
\address{\'Alvaro del Pino, Utrecht University, Department of Mathematics, Budapestlaan 6, 3584 Utrecht, The Netherlands}
\email{a.delpinogomez@uu.nl}

\begin{abstract}
We define billiards in the context of sub-Finsler Geometry. We provide symplectic and variational (or rather, control theoretical) descriptions of the problem and show that they coincide. We then discuss several phenomena in this setting, including the failure of the reflection law to be well-defined at singular points of the boundary distribution, the appearance of gliding and creeping orbits, and the behavior of reflections at wavefronts.

We then study some concrete tables in $3$-dimensional euclidean space endowed with the standard contact structure. These can be interpreted as planar magnetic billiards, of varying magnetic strength, for which the magnetic strength may change under reflection. For each table we provide various results regarding periodic trajectories, gliding orbits, and creeping orbits.
\end{abstract}

\maketitle
\setcounter{tocdepth}{1}
\tableofcontents



%

\section{Introduction} \label{sec:introduction}

Given a smooth manifold $M$, a \textbf{distribution} $\xi \subset TM$ is a (vector) sub-bundle of the tangent bundle. The dimension of its fibres is called the \textbf{rank}, which we denote by $\rank(\xi)$. Distributions possess many local differential invariants that measure their non-involutivity with respect to the Lie bracket. Namely, we define the Lie flag as the sequence of $C^\infty$-modules of vector fields:
\[  \Gamma^{(1)}(\xi) \subset  \Gamma^{(2)}(\xi) \subset \cdots  \subset \Gamma^{(i)}(\xi) \subset \cdots \]
where $\Gamma(\xi) := \Gamma^{(1)}(\xi)$ is the space of vector fields tangent to $\xi$ and $\Gamma^{(i+1)}(\xi) := [\Gamma^{(i)}(\xi),\Gamma(\xi)]$ is the space of vector fields that are $C^\infty$ combinations of $i$ Lie brackets with entries in $\Gamma(\xi)$. If there exists an $i_0$ such that $\Gamma^{(i_0)}(\xi) = \Gamma(TM)$, we say that $\xi$ is \textbf{bracket-generating}.

The bracket-generating condition states that any infinitesimal motion is a combination of motions tangent to $\xi$. This statement has a global analogue: A classic theorem of Chow and Rashevskii states that any two points in a bracket-generating manifold $(M,\xi)$ can be connected by a curve tangent to $\xi$ (often called an \textbf{horizontal curve}). With this result at hand we find it natural to study horizontal curves that minimise distance in the following sense.


A \textbf{sub-Finsler manifold} $(M,\xi,f)$ is a manifold endowed with a bracket-generating distribution and a fibrewise Finsler norm $f: \xi \to \RR$. We allow $f$ to be asymmetric. As a special case, one could equip $\xi$ with a Riemannian metric $g$ and choose $f(v) := \sqrt{g(v,v)}$. The resulting structure $(M,\xi,g)$ is called a \textbf{sub-Riemannian manifold}. A sub-Finsler manifold naturally carries a (non-reversible) metric, called the \textbf{Carnot-Carath\'eodory metric}, given by the infimum of the lengths of the connecting horizontal paths:
\[ d(x,y) := \inf\, \{ \textrm{length}_f(\gamma) \mid \dot\gamma\in \xi, \gamma(0)=x, \gamma(1)=y \}. \]
In order to measure the length, it is necessary to impose some regularity on $\gamma$. All throughout the text, we will choose $\gamma$ to be Lipschitz and therefore almost everywhere differentiable with uniformly bounded differential.

As stated above, one can then study geodesics for this metric. It turns out that this (almost!) works as in the usual Riemannian setting: there is a set of well-behaved geodesics that admits a cotangent description (namely, they are projections of Hamiltonian orbits of the Legendre dual of $f$). We restrict our attention to these. The geodesics that are left out are the so-called \emph{strict abnormals}  \cite{ASr,BFPR,Mon2,Z}: understanding them better is a central question in sub-Riemannian geometry (but that we will mostly ignore). 

\begin{remark}
The Legendre dual of $f$ is a Hamiltonian $T^*M \to \RR$ that is invariant under translation by the annihilator of $\xi$. The theory of more general Hamiltonians of this form is to our knowledge still unexplored. \hfill$\triangle$
\end{remark}

A \textbf{sub-Finsler billiard table} is a closed equidimensional submanifold $U \subset (M,\xi,f)$ with boundary (and possibly corners). Much like in classical billiards, one can pose questions about periodic trajectories, integrability, caustics, and gliding orbits. Our goal in this paper is to settle some foundational questions about the theory and to work out some concrete examples showcasing exotic behaviors with respect to classical billiards. Above all, we hope to convince the reader that this is a natural setup with many intriguing open questions.

\subsection{Results and structure}

In order to study sub-Finsler geodesics and billiard trajectories we need some background from Geometric Control Theory. Section \ref{sec:controlTheory} provides an overview of all the tools we need, as well as sketches of proof of many key statements. In Section \ref{sec:sub-FinslerSystems} we introduce sub-Finsler manifolds and discuss how the general theory applies to them.

In Section~\ref{sec:billiardFlow} we introduce sub-Finsler billiard dynamics and discuss some of their properties. Using a symplectic viewpoint (Subsection~\ref{ssec:symplecticViewpoint}), we show how one may define the billiard problem (in particular, the reflection law) in terms of the canonical dynamics of a certain piecewise hypersurface in cotangent space naturally associated to the sub-Finsler norm. In Proposition~\ref{prop:ReflectionIsSymplectomorphism} we show that the reflection law defines, away from the singularities, a symplectomorphism on the space of sub-Finsler geodesics. We then study (Subsection \ref{ssec:variationalViewpoint}) the dynamics in control theoretical terms (i.e., the sub-Finsler analogue of the usual variational approach), showing that the reflection law arises naturally from a minimisation problem. Both definitions agree, which we show in Proposition~\ref{prop:sympIsVariational}.

In this Section we include as well a discussion of the gliding phenomenon for orbits (Subsection~\ref{sssec:glide}). The desired claim is that gliding orbits are geodesics of the sub-Finsler structure in the boundary of the table. In Proposition~\ref{prop:boundaryGeodesic} we prove this assuming $C^1$-convergence. The converse statement (every boundary geodesic is a limit of billiard trajectories), is studied in some concrete cases in Section \ref{sec:specialCase}.

In Section \ref{sec:specialCase} we study several billiard tables in the standard contact structure in $\RR^3$, endowed with the standard flat metric lifted from the plane. When the table is a cylinder with circular base (Subsections \ref{ssec:verticalWalls} and \ref{ssec:standardCylinder}), the billiard system is integrable and we can provide a complete characterisation of the periodic, gliding, and creeping orbits. We then look at tables where the boundary consists of one (Subsection \ref{sec:horizontalPlane}) or two disjoint planes (Subsection \ref{ssec:horizontalBand}), or a finite cylinder with planar top and bottom boundaries (Subsection \ref{ssec:finiteCylinder}), providing some partial results about their dynamics. A takeaway message of these computations is that complicated behaviors can appear as geodesics approach a singularity between the distribution and the boundary of the table.

\section{An overview of control theory} \label{sec:controlTheory}

In this paper we are interested in sub-Finsler geodesics, which we study using a control theoretical approach. In this Section we provide some general background that we will specialise to the sub-Finsler setting in Section \ref{sec:sub-FinslerSystems}.

\emph{Disclaimer:} In order to make the article accessible to researchers in Symplectic Geometry and Dynamics, we have decided to present many key ideas from Control Theory in a fair amount of detail. In doing so, we have emphasised a geometric viewpoint, which borrows from the excellent reference \cite{AS}. Control theorists should feel free to jump ahead to the next Section.

\subsection{Control systems}

\begin{definition} \label{def:controlSystem}
Given a smooth manifold $M$, a \textbf{control system} over $M$ is:
\begin{itemize}
\item A locally trivial fibre bundle $C \to M$ whose fibres are manifolds (possibly with boundary or corners).
\item A bundle map $\rho: C \to TM$ called the \textbf{anchor}.
\end{itemize}

Let $I \subset \RR$ be a (possibly infinite) interval. A \textbf{control} of $C$ is an $I$-family of sections
\[ (u_t: U_t \quad\to\quad C)_{t \in I} \]
that is smooth in space $M$, $L^\infty$ in time $I$, and whose open domain $U_t \subset M$ varies smoothly in $t$.
\end{definition}
We write $C^\infty(M,C)$ for the sheaf (over $M$) of autonomous controls and $\Gamma(M \times I, C)$ for the sheaf (over $M \times I$) of time-dependent controls. Note that if both $M$ and $C$ have non-trivial topology it may be the case that $C^\infty(M,C)$ has no global sections (which is the reason why we must work with smoothly varying domains for the controls). 

Given a control system, we can look at its solution curves, and at the points that can be reached with them:
\begin{definition}
A Lipschitz curve $\gamma: [0,T] \to M$ is \textbf{admissible} if there is a control $u_t$ such that $\dot\gamma(t) = (\rho \circ u_t)(\gamma(t))$ almost everywhere.

Given a point $p \in M$, we define its \textbf{attainable sets} at times exactly $T \in [0,\infty)$, at most $T$, and at infinity:
\begin{align*}
\Acal_T(p) & \quad:= \quad \{q \in M \mid \exists \gamma\text{ admissible } \gamma(0)=p,\gamma(T)=q\}, \\
\Acal_{\leq T}(p) &\quad := \quad \cup_{t \in [0,T]} \Acal_t(p), \\
\Acal_{\leq \infty}(p) &\quad := \quad \cup_{t \in [0,\infty)} \Acal_t(p).
\end{align*}
\end{definition}
\begin{remark}
Observe that $\Acal_{\leq t_1}(p) \subset \Acal_{\leq t_2}(p)$ holds for $t_1 \leq t_2$, but the analogous inclusion for attainable sets at a concrete time is not necessarily true (for instance it fails for systems involving a drift term). \hfill$\triangle$
\end{remark}

\begin{remark}
The boundary of the attainable set is a subset of the \textbf{wavefront}, cf.~Definition~\ref{def:wavefront}. \hfill$\triangle$
\end{remark}

\subsubsection{On our assumptions on $C$}

Write $\Vertical(C) \to C$ for the fibre tangent bundle of C. At an element $u \in C$ lying over $p \in M$, we can consider the vector space $\Vertical_u C = T_uC_p$. In the presence of boundary/corners, not all directions in $\Vertical_u C$ are meaningful as appropriate variations of the control. This leads us to define the \textbf{inward-pointing tangent space} $\Vertical^- C$ to be the collection of vectors that are realised as $\dot\gamma(0)$ by a short piece of smooth curve $\gamma:[0,\varepsilon)\to C$ tangent to a fibre. We remark that its fibres are corners in each $\Vertical_u C$, and thus \emph{cones} (i.e., an $\RR^+$-invariant subsets).

If we take sections, a similar phenomenon occurs. $C^\infty(M,C)$ is a Fr\'echet manifold if $C$ is without boundary. In the general setting, it is still Fr\'echet (as can be shown by adding to $C$ a boundary collar, effectively extending it to a fibration by smooth open manifolds), but has corners of arbitrarily large codimension.

\subsubsection{Linearising}

Given a control system $\rho: C \to TM$ and a control $(u_t)_{t \in [0,T]}$, we want to describe the tangent space $T_u\Gamma(M \times I, C)$ to $\Gamma(M \times I, C)$ at $u$ (i.e., the space of infinitesimal variations). 
\begin{lemma}
Infinitesimal variations of $(u_t: U_t \subset M \to C)_{t \in I}$ correspond to sections
\[ (v_t: U_t \subset M \to \Vertical_{u_t}C)_{t \in I} \]
that are smooth in $M$ and $L^\infty$ in $I$.
\end{lemma}
\begin{proof}
Given such a $v_t$, we can use the exponential map (for some metric in $C$) to produce a variation, i.e., a family of controls $V_{s,t}: U_t \to C$ with $V_{0,t} = u_t$ and $\partial_s V_{s,t} = v_t$. Conversely, given $V_{s,t}$ we recover the infinitesimal variation as $v_t = \partial_s V_{s,t}$.
\end{proof}

We may now say that a variation at $u = (u_t)_{t \in I}$ is \textbf{inward-pointing} if it takes values in $\Vertical_{u_t}^-C$. The subsheaf of inward-pointing variations based at $u$ is denoted by $T_u^-\Gamma(M \times I,C)$. Its sections form a cone, but not necessarily a vector space.

We can \emph{linearise} our control system along the time-dependent control $u$. We think of $\Vertical_u^-C$ as a bundle over $M \times I$ that is only $L^\infty$ in time (since it arises as the pullback by an $L^\infty$-map of an actual vector bundle over $C \times I$):
\begin{definition} \label{def:linearisedAnchor}
The \textbf{linearised anchor map} along $u$ is the composition:
\[ \Vertical_u^-C \to T_{\rho(u)}TM\times I \to TM\times I, \]
where the first arrow is the differential $d_u\rho$ and the second arrow is given by the tautological identification between the vector spaces $T_{\rho(u)}T_qM$ and $T_qM$.
\end{definition}
We will use bundles and control systems that are only $L^\infty$ in time in subsequent subsections as well. In order not to clutter the text, further details regarding these non-autonomous setting can be found in Appendix \ref{appendix:nonAutonomous}.

\subsection{The endpoint map} \label{ssec:endpoint}

Given two points $p,q \in M$ and a control system $\rho: C \to TM$, we want to look at the admissible curves that connect $p$ with $q$, particularly those that do so in the least amount of time. The following definition and subsequent discussion clarifies their nature:
\begin{definition}
The \textbf{endpoint map} of $\rho: C \to TM$ at time $T$ with basepoint $p$ is the smooth map:
\begin{align*}
\Endpoint_p^T: \quad \Gamma(M \times [0,T], C) \quad\to\quad & M \\
                    u_t                         \quad\to\quad & \gamma(T),
\end{align*}
where $\gamma: [0,T] \to M$ is the (unique) Lipschitz solution with control $u_t$ and initial point $p$.

Similarly, the \textbf{endpoint map} of $\rho: C \to TM$ for time at most $T$ reads:
\begin{align*}
\Endpoint_p^{\leq T}: \quad \Gamma(M \times [0,T], C) \times [0,T] \quad\to\quad & M \\
                    (u_t,t_0)                         \quad\to\quad & \gamma(t_0).
\end{align*}
\end{definition}

\begin{assumption}
The map $\Endpoint_p^T$ (resp. $\Endpoint_p^{\leq T}$) is of interest when we look at attainable sets $\Acal_T(p)$ with arrival time exactly $T$ (resp. attainable sets $\Acal_{\leq T}(p)$ with arrival at most $T$). In the sub-Finsler setting there is no difference between $\Acal_T(p)$ and $\Acal_{\leq T}(p)$, as we shall see. For this reason, and in order to simplify the discussion, we henceforth focus on $\Acal_T(p)$ and $\Endpoint_p^T$. \hfill$\triangle$
\end{assumption}

\begin{remark}
Recall that $\Gamma(M \times [0,T], C)$ is a sheaf and some controls are not defined over the whole of $M$. As such, a control may not actually produce a solution curve over the complete interval $[0,T]$. It follows that the endpoint map is only defined for a subset of the controls. This is a technical point that makes no difference in practice. \hfill$\triangle$
\end{remark}

Assuming that $C$ is without boundary, the sheaf of controls $\Gamma(M \times [0,T], C)$ takes values in Fr\'echet manifolds. One may then ask whether the subspace of controls connecting $p$ with $q$ (i.e., the fibre of $\Endpoint_p^T$ over $q$) is also a manifold. We are thus led to studying whether the differential $d_u\Endpoint_p^T$ of the endpoint map at a control $(u_t)_{t \in [0,T]}$ is surjective.

If $C$ has boundary/corners, it is still meaningful to look at infinitesimal variations of the control, but these should be inward-pointing. Due to this, we concentrate on:
\begin{definition}
The \textbf{infinitesimal endpoint map} at $u$, starting at $p$, and at time $T$ is the restriction:
 \[d_u^-\Endpoint_p^T \quad:=\quad d_u\Endpoint_p^T|_{T_u^-\Gamma(M \times [0,T],C)}. \]
\end{definition}

\begin{lemma} \label{lem:infEndpointMap}
Let $u_t$ be a control with $\gamma$ its solution starting at $p$ and $\phi_t$ its flow. Then, the infinitesimal endpoint map is given by the expression:
\begin{align*}
d_u\Endpoint_p^T: \quad T_u^-\Gamma(M \times [0,T],C) &\quad\to\quad  T_{\gamma(T)}M, \\
d_u\Endpoint_p^T(v) &\quad=\quad 
\int_0^T d\phi_{t \to T} \circ d\rho \circ v(t)dt \end{align*}
\end{lemma}

Here we use the notation $\phi_{t \to T}$ to denote the flow between time $t$ and $T$. 

\begin{proof}
Using the flow $\phi_t$ we act on $(C,\rho)$ by pushforward,  producing a new control system. This transformation identifies each solution of $u_t$ with a constant trajectory, allowing us to carry out all computations in a single tangent space. Let us elaborate.

Let $\nu$ be an admissible curve of $(C,\rho)$. The velocity vector of the pushforward trajectory $\phi_{t \to T} \circ \nu(t)$ reads:
\[ (\phi_{t \to T} \circ \nu)'(t) = d_{\nu(t)}\phi_{t \to T}(\nu'(t) - \rho \circ u_t \circ \nu(t)), \]
showing that, if $\nu$ was a solution of a control $v_t$, it is now a solution of the pushforward control
\[ (\phi_{t \to T})_*v_t(q) := d_{\phi_t(q)}\phi_{t \to T} \circ (\rho \circ  v_t - \rho \circ u_t)(\phi_{t \to T}^{-1}(q)). \]
This leads us to define a time-dependent control system:
\[ (D(q,t) := C(\phi_{t \to T}^{-1}(q)), \quad \rho_t(d) := d\phi_{t \to T} \circ (\rho \circ d - \rho \circ  u_t)). \]
We remark that this system is only $L^\infty$ in time and, in particular, as a bundle over $M \times [0,T]$, $D$ is not necessarily smooth. The pushforward $(\phi_{t \to T})_*$ provides a 1-to-1 correspondence between $(C,\rho)$ and $(D,\rho_t)$ for admissible curves, controls, and infinitesimal variations of controls.

In particular, the solution $\gamma$ of $u_t$ is mapped to the constant curve $t \to \gamma(T) = q$. The variations of the latter are $L^\infty$-maps $[0,T] \to T_qM$ with image in the time dependent family of subsets:
\[ E_t := d\rho_t(\Vertical_{(\phi_{t\to T})_*u_t}^-D_{q,t}) = d\phi_{t \to T} \circ d\rho\, (\Vertical^-_{u_t}C_{\gamma(t)}). \]
The image of such a variation under the differential of the endpoint map of $(D,\rho_t)$ is given by integrating in $t$. Translating back to $C$ yields the claim.
\end{proof}
\begin{remark}
The reader should note that the image of the infinitesimal endpoint map at $u$ remains the same if we replace $C$ by its linearisation at $u$ (as in Definition \ref{def:linearisedAnchor}). \hfill$\triangle$
\end{remark}

 \subsubsection{Critical curves}

As stated before, we are interested in those curves at which the endpoint map fails to be submersive:
\begin{definition}
Fix a point $p \in M$ and a time $T$. A control $(u_t)_{t \in [0,T]}$ is:
\begin{itemize}
\item \textbf{critical} if $d_u^-\Endpoint_p^T$ is not surjective.
\item \textbf{minimising} if $\gamma(T) \in \partial\Acal_T(p)$, where $\gamma: [0,T]$ is the solution of $u$ with initial point $p$.
\item \textbf{minimising locally in time} if it is minimising over any sufficiently small interval.
\item a \textbf{local minimiser} if there is no control $(v_t)_{t \in [0,T']}$, $T' < T$, $L^\infty$-close to $u$, such that the solution $\nu$ of $v$ starting at $p$ satisfies $\nu(T') = \gamma(T)$.
\end{itemize}
\end{definition}

In the proof of Lemma \ref{lem:infEndpointMap} we already used implicitly:
\begin{lemma}
Let $(u_t)_{t \in [0,T]}$ be a control with solution $\gamma$ and let $(v_t)_{t \in [0,T]}$ be another control satisfying $u_t(\gamma(t)) = v_t(\gamma(t))$. Then, $u_t$ is critical (at $\gamma$) if and only if $v_t$ is critical.
\end{lemma}
Do note that the same conclusion does not follow if we instead assume the weaker condition $\rho \circ u_t(\gamma(t)) = \rho \circ v_t(\gamma(t))$. 

We then speak of admissible curves $\gamma$ being \textbf{critical}/\textbf{minimising} when they are solutions of critical/minimising controls. Since we considered the restricted linearisation of the endpoint map, it follows that:
\begin{lemma}
A local minimiser is critical.
\end{lemma}

\subsection{Filippov's theorem}

A central question in Control Theory is whether the infimal time of arrival between two points is realised by a minimising curve. The following statement \cite[Subsection 10.3]{AS} gives a sufficient criterion (and one can easily produce examples where the conclusions fail to hold if the compactness or convexity assumptions are dropped):
\begin{proposition}
Let $\rho: C \to TM$ be a control system with fibrewise compact and convex image. Then:
\begin{itemize}
\item The attainable sets of $C$ are compact.
\item Assume additionally that $\Acal_\infty(p)=M$ and let $T = \inf \{t \mid q\in\Acal_t(p)\}$. Then, there exists a (possibly not unique) admissible curve connecting $p$ with $q$ with arrival time $T$.
\end{itemize}
\end{proposition}

We provide now an equivalent claim (with its corresponding proof), since we will need it later on for our discussion about billiards:
\begin{lemma} \label{lem:LipschitzLimit}
Let $M$ be compact. Let $\rho: C \to TM$ be a control system with fibrewise compact and convex image. Then:
\begin{itemize}
    \item any sequence of admissible curves $\{\gamma_i\}_{i=1,2,\cdots}$ has a subsequence converging in $C^0$ to a Lipschitz curve $\gamma_\infty$.
    \item $\gamma_\infty$ is an admissible curve.
\end{itemize} 
\end{lemma}
\begin{proof}
Since $\rho(C)$ is fibrewise compact, all admissible curves are Lipschitz with the same Lipschitz constant. Compactness of $M$ shows that any sequence of curves is $C^0$-bounded. Arzela-Ascoli then tells us that there is a converging subsequence and that the limit $\gamma_\infty$ is Lipschitz (with the same constant).

For the second claim, we work locally in time and space: We pick some $t \in I$ in the domain of all our curves and a little ball $B$ containing $\gamma_\infty(t) \in M$. Using the identification $TB \cong B \times \RR^n$ we can regard all $\rho(C)_p$, $p \in B$, as subsets of $\RR^n$. Then, we observe that the inclusion holds:
\[ \gamma_i(t+h)-\gamma_i(t) = \int_t^{t+h} \gamma_i'(s)ds \quad\in\quad \textrm{ConvexHull}\left(\bigcup_{t \leq s \leq t+h} \, \rho(C)_{\gamma_i(s)}\right), \]
since $\gamma'_i(s) \in \rho(C)_{\gamma_i(s)}$. By taking the limit we deduce:
\[ \gamma_\infty(t+h)-\gamma_\infty(t) \quad\in\quad \textrm{ConvexHull}\left(\bigcup_{t \leq s \leq t+h} \, \rho(C)_{\gamma_\infty(s)}\right), \]
which implies that, if $t$ is a point of differentiability:
\[ \gamma_\infty'(t) \in \textrm{ConvexHull}\left(\rho(C)_{\gamma_\infty(t)}\right) = \rho(C)_{\gamma_\infty(t)}. \]
\end{proof}

The same idea allows one to prove that trajectories of a control system can be used to approximate trajectories of the control system given by its convex hull. This is the celebrated \emph{relaxation theorem} of Filippov, developed independently by M. Gromov under the name of \emph{convex integration}:
\begin{proposition} \label{prop:Filippov}
Let $\rho: C \to TM$ be a control system with fibrewise compact image. Let $D \subset TM$ be the fibrewise convex hull of $\rho(C)$. Then, the attainable sets of $D$ are the closures of the attainable sets of $C$.
\end{proposition}
Do note that $D$ can be realised as the image of $C \times C \times [0,1]$ under the convex combination of two copies of $\rho$, so it fits into our formalism. These results tell us that, from a control theory perspective, we may as well restrict our attention to control systems that are fibrewise convex (but possibly non-smooth and non-strictly convex).

\subsection{The Pontryagin maximum principle}

We will now provide a cotangent characterisation of critical and minimising curves. As we shall explain later, this generalises the interpretation of the Riemannian (co)geodesic flow as a Hamiltonian flow. We refer the reader to \cite[Chapter 12]{AS}.

\subsubsection{Cotangent lifts} \label{sssec:cotangentLift}

Let $(u_t)_{t \in [0,T]}: M \to C$ be a control. Its image $(\rho \circ u_t)_{t \in [0,T]}$ under the anchor map is a time-dependent vector field, which can thus be regarded as a fibrewise linear, time-dependent Hamiltonian:
\begin{align*}
H_{u_t}: \quad T^*M \quad\to\quad & \RR \\
\alpha \quad\to\quad & H_{u_t}(\alpha) := \alpha(\rho \circ u_t).
\end{align*}
\begin{definition} \label{def:cotangentLiftA}
The time-dependent Hamiltonian vector field $X_{u_t}$ corresponding to the Hamiltonian functions $H_{u_t}$ is said to be a \textbf{cotangent lift} of $u_t$.
\end{definition}
Do note that $X_{u_t}$ is smooth in space, but only measurable in time. Similarly:
\begin{definition} \label{def:cotangentLiftB}
Let $\gamma:[0,T]\to M$ be an admissible curve with control $u_t$ and initial point $p \in M$. A curve $\lambda:[0,T] \to T^*M$ is a \textbf{cotangent lift} of $\gamma$ if:
\begin{itemize}
\item $\lambda(t) \neq 0$,
\item $\pi \circ \lambda(t) = \gamma(t)$,
\item $\dot\lambda(t) = X_{u_t}(\lambda(t))$.
\end{itemize}
\end{definition}
I.e., the lifted curve $\lambda$ is a momentum for $\gamma$. 

\subsubsection{Cotangent characterisation of critical curves} \label{sssec:PSP}

Let us introduce some notation: A covector $\lambda \in V^*$ \textbf{supports} a subset $A \subset V$ if $\lambda|_A \leq 0$.

Fix a point $p \in M$ and a critical control $(u_t)_{t \in [0,T]}$ with solution $\gamma(t)$. Due to criticality, the image of $d_u^-\Endpoint_p^T$ is not the whole of $T_{\gamma(T)}M$. It is a cone, and we shall see that is is a convex set. From this, it follows that there exists a covector $\lambda(T) \in T_{\gamma(T)}^*M$ of support, i.e., no infinitesimal variation of the control allows us to move in the codirection $\lambda(T)$.

The covector $\lambda(T)$ can be pulled back using the flow of $u_t$, yielding a cotangent lift $\lambda: [0,T] \to T^*M$ of $\gamma$. Then, criticality can be read in terms of the lift $\lambda$ as follows:
\begin{proposition} \label{prop:PSP}
Fix a point $p \in M$ and a control $u: M \times [0,T] \to C$. Then, the following conditions are equivalent:
\begin{itemize}
\item $u$ is critical,
\item $u|_{[0,t]}$ is critical for all $t \in [0,T]$,
\item there exists a cotangent lift $\lambda:[0,T] \to T^*M$ such that, for almost all $t \in [0,T]$, the $1$-form $\rho^*\lambda(t)$ supports the inward-pointing tangent space $\Vertical_{u_t(\gamma(t))}^-C$.
\end{itemize}
\end{proposition}

If $C$ has no boundary/corners, the last condition means that the linear function $\rho^*\lambda(t): C_{\gamma(t)} \to \RR$ has a critical point at $u_t(\gamma(t))$.

\begin{proof}[Sketch of proof]
Let $\phi_t$ be the flow of $u_t$. As in Lemma \ref{lem:infEndpointMap}, we translate the problem to the time-dependent control system
\[ (D(q,t) := C(\phi_{t \to T}^{-1}(q)), \quad \rho_t(d) := d\phi_{t \to T} \circ (\rho \circ d - \rho \circ u_t)). \]
In order to study $\gamma$, we look at its pushforward, the constant function $t \to q = \gamma(T)$. We determined that its variations are $L^\infty$-maps:
\[ t \in [0,T] \quad\to\quad E_t := d\rho_t(\Vertical_{(\phi_{t\to T})_*u_t}^-D_{q,t}) = d\phi_{t \to T} \circ d\rho\, (\Vertical^-_{u_t}C_{\gamma(t)}) \subset T_qM. \]

Let us denote $\EE_t$ for the image of the endpoint map of $(D,\rho_t)$ at some intermediate time $t \in [0,T]$. Morally, we want to claim that:
\[ \EE_t = \textrm{ConvexHull}\left(\cup_{s \in [0,t]} \,E_s \right). \]
The idea behind this identity is to take \emph{needle variations}, i.e., sequences of variations whose support concentrates at a particular $s$, varying the endpoint exactly in a direction contained in $E_s$. By taking the support of these variations to be very small time regions, we can add them, effectively producing any convex combination.

There is a caveat in this argument: The variations we consider are $L^\infty$ in time, so the infinitesimal directions at a concrete $t$ are irrelevant. We should then define: A vector $w \in E_t$ is \emph{Lebesgue} if there is an $L^\infty$-section $v: \Op(t) \to T_qM$ such that $v(s) \in E_s$ and $\int_{t-\delta}^{t+\delta} w- v(s) ds = O(\delta)$. I.e., these are the vectors admitting needle approximations. The correct statement reads then:
\[ \EE_t = \textrm{ConvexHull}\left(\cup_{s \in [0,t]} \, \textrm{Lebesgue}\left(E_s\right)\right). \]
From this identity, we deduce:
\begin{itemize}
    \item the constant curve $q$ is critical if and only if there exists a covector $\lambda \in T^*_qM$ supporting $\EE_T$.
    \item $\lambda$ also supports all prior images $\EE_t$, $t \in [0,T]$, and thus the sets $\EE_t$ are non-decreasing in $t$.
    \item $\lambda$ supports $E_t$ for almost all $t$.
\end{itemize}
Translating these statements back to $C$ provides a complete description of the image of the infinitesimal endpoint map and concludes the proof.
\end{proof}

We are thus lead to define:
\begin{definition}
Let $(u_t)_{t \in [0,T]}$ be a control with solution $\gamma$. Then, a cotangent lift $\lambda$ of $\gamma$ is said to be \textbf{critical}, if it satisfies the third property in Proposition \ref{prop:PSP}.
\end{definition}

\subsubsection{The Pontryagin maximum principle}

In the case of minimisers, the previous result can be strengthened to yield the Pontryagin maximum principle (PMP):
\begin{proposition}[PMP] \label{prop:PMP}
Fix a point $p \in M$ and a control $u: M \times [0,T] \to C$ with solution $\gamma$ starting at $p$. If $\gamma$ is a minimiser, there exists a cotangent lift $\lambda:[0,T] \to T^*M$ satisfying:
\[ H_{u_t}(\lambda(t)) = \max_{\rho(v) \in C_{\gamma(t)}} \lambda(t)(v). \]
\end{proposition}
\begin{proof}[Sketch of proof]
As before, we act on the system by pushforward using the flow $\phi_t$ of $u_t$. However, instead of linearising (i.e., considering infinitesimal variations of the constant function $t \to q = \gamma(T)$), we look directly at the pushforward system $(D,\rho_t)$. Namely, we define:
\[ B_t := \rho_t(D_{q,t}) = d\phi_{t \to T} (\rho\, (C_{\gamma(t)}) - \rho\, (u_t(\gamma(t)))) \subset T_qM. \]
If $C$ is not fibrewise linear (we invite the reader to visualise $C$ as being convex), the sets $B_t$ contain more information than the sets $E_t$ used in Proposition \ref{prop:PSP}.

Given the $(D,\rho_t)$-attainable set $\Acal_T(q)$, we define $T\Acal_T(q)$ as the union of all limits $\lim_{a \to q} b\frac{q-a}{|q-a|}$ with $a \in \Acal_T(q)$ (for some metric in $M$) and $b \in \RR^+$. Even though the attainable set may not be smooth, this serves as a linear approximation of it at $q$. The point now is that, since $\gamma$ is a minimiser, there exists a covector $\lambda$ in $q$ supporting $T\Acal_T(q)$.

We claim that any covector $\lambda \in T_q^*M$ supporting $T\Acal_T(q)$ supports also the set:
\[ A_T := \cup_{t \in [0,T]} \textrm{Lebesgue}(B_t), \]
and, in particular, for almost every $t$, $\lambda|_{B_t}$ attains a maximum at zero. We will prove this using needle variations. We remark that, unlike the variations used in Proposition \ref{prop:PSP}, which were actual infinitesimal variations, the ones we will use now are not.

Indeed: Suppose $v \in T_qM$ is a Lebesgue vector of some $B_t$. Let $v(s) \in B_s$ be a family satisfying $\lim_{\delta \to 0} \int_{t-\delta}^{t+\delta} v(s) ds = v$. This means that we can produce a control $u^\delta$ that agrees with $u$ everywhere except in an interval of size $2\delta$, where it is instead given by the family $v(s)$. When we integrate this control, it will produce a trajectory whose endpoint is roughly $2\delta v$. As such, by taking $\delta \to 0$ we obtain a sequence of controls producing motion in the direction of $v$. Our assumptions then imply that $\lambda$ must support $v$, concluding the claim for the pushforward system. Translating back to $C$ concludes the proof.
\end{proof}

\begin{definition}
Let $u_t$ be a control with solution $\gamma$. A cotangent lift $\lambda$ of $\gamma$ is said to \textbf{satisfy PMP} if it satisfies the property in the statement of Proposition \ref{prop:PMP}.
\end{definition}
In general, a curve may have a lift satisfying PMP but not be minimising due to the higher order behaviour of the control system at $u$.

\begin{remark}
Observe that the sequence $u^\delta$ used in the proof of Proposition \ref{prop:PMP} is bounded and differs from $u$ only in an interval of size $2\delta$. As such, $u^\delta \to u$ in the $L^p$-topology, as long as $p < \infty$. It follows that local minimisers (in $L^p$, $p < \infty$) satisfy PMP as well. \hfill$\triangle$
\end{remark}
\begin{remark} \label{rem:Linfinity}
The previous reasoning does not apply to $L^\infty$ since, using the notation from the proof, $|u^\delta - u| = |v|$ for all $\delta$. Instead, if $u$ is a local minimiser in the $L^\infty$-topology, we have to restrict our attention to controls that a.e. approach $u$. This implies that we only see the control system locally around $u$. We can use the same argument to then prove that there is a cotangent lift $\lambda$ such that $\lambda|_C$ attains a local (instead of global) maximum at $u$ for almost all values of $t$. \hfill$\triangle$
\end{remark}

\subsection{Maximised Hamiltonians and abnormal curves} \label{ssec:maximisedHamiltonians}

We explained before how a control system may be lifted to the cotangent bundle in a Hamiltonian manner. We have now the tools to take this further and show that (under certain assumptions) there is a single Hamiltonian system whose solutions project down to minimisers (but not all minimisers arise in this manner, as we shall see). First we observe:
\begin{corollary}\label{Cor:preservedQuantity}
Let $u_t$ be a control with solution $\gamma$ and cotangent lift $\lambda$ satisfying PMP. Then, the function $t \to H_{u_t}(\lambda(t))$ is constant.
\end{corollary}
\begin{proof}
Whenever we can differentiate in $t$, it follows from the Cartan identity that:
\[ \dfrac{d}{dt}H_{u_t}(\lambda(t)) = dH_{u_t}(\dot\lambda(t)) + \dot H_{u_t}(\lambda(t)) = dH_{u_t}(X_{u_t}) + \lambda(t)(\rho(\dot u_t)) = 0, \]
where the second term vanishes due to Proposition \ref{prop:PSP}. The general argument is similar.
\end{proof} 

Then, Proposition \ref{prop:PMP} suggests us to introduce:
\begin{definition}
The \textbf{maximised Hamiltonian} defined by the control system $\rho: C \to TM$ is
\begin{align*}
H_\rho: \quad T^*M              \quad\to\quad & \RR \\
\lambda    \quad\to\quad & H_\rho(\lambda) := \max_{v \in \rho(C_q)} \lambda(v).
\end{align*}
\end{definition}
We note that $H_\rho$ is positively homogeneous (of degree one) and only depends on the convex hull of $\rho(C)$. In general, it is not smooth.

\subsubsection{Characteristics}

We now focus on the cotangent region in which $H_\rho$ is smooth. Before we continue, let us recall some notation:
\begin{definition}
Let $\Sigma \subset (T^*M,\omega)$ be a smooth submanifold. Its \textbf{characteristic foliation} is $\Fcal_\Sigma := \ker(\omega|_\Sigma)$. Vectors and submanifolds of $\Sigma$ are said to be \textbf{characteristic} if they are tangent to $\Fcal_\Sigma$.
\end{definition}
The rank of the $2$-form $\omega|_\Sigma$ may vary from point to point, so we think of $\Fcal_\Sigma$ as a singular distribution. Whenever $\omega|_\Sigma$ has constant rank locally, $\Fcal_\Sigma$ is indeed involutive due to the fact that $\omega$ is closed (hence our usage of the word ``foliation''). We recall:
\begin{lemma}
The Hamiltonian trajectories of $H_\rho$, up to reparametrisation, are in correspondence with the characteristic trajectories of its level sets.
\end{lemma}

Particularising this to our setting, we have the following:
\begin{lemma}
Let $\gamma$ be an admissible curve with lift $\lambda$ contained in the region where $H_\rho$ is smooth. Then, the following statements are equivalent:
\begin{itemize}
    \item[i. ] $\lambda$ satisfies PMP.
    \item[ii. ] $\lambda$ is a characteristic curve in a level set of $H_\rho$.
\end{itemize}
\end{lemma}
\begin{proof}
By definition, the inequality $H_\rho - H_{u_t} \geq 0$ holds for any control $u_t$. Additionally, the equality $H_\rho - H_{u_t} = 0$ along a curve $\lambda$ with control $u_t$ is equivalent to PMP. If $H_\rho - H_{u_t}$ vanishes along $\lambda$, it follows that $\lambda$ is a Hamiltonian trajectory of $H_\rho$ if and only if it is a trajectory of $H_{u_t}$. This is because Hamiltonian orbits are determined by first jet data. As such, Condition (i) implies (ii). The converse follows by taking as control any $u_t$ maximising $\lambda(u_t)$.
\end{proof}

\subsubsection{Abnormal curves}

That is, some of the curves that are minimising (to first order) admit a nice cotangent description in terms of $H_\rho$. Some others do not, leading us to define:
\begin{definition}
A minimiser $\gamma$ is:
\begin{itemize}
    \item \textbf{abnormal}, if it has a cotangent lift $\lambda$ satisfying PMP and contained in the non-smooth locus of $H_\rho$.
    \item \textbf{strictly abnormal}, if all its lifts satisfying PMP are in the non-smooth locus of $H_\rho$.
    \item \textbf{normal}, otherwise.
\end{itemize}
\end{definition}
In particular, strict abnormals are not projections of characteristics of $H_\rho$. Understanding the nature of strictly abnormal minimisers is one of the central questions in sub-Riemannian Geometry. Here, we concentrate on the normal minimisers.

\section{Sub-Finsler control systems} \label{sec:sub-FinslerSystems}

Recall that a \textbf{sub-Finsler manifold} is a triple $(M,\xi,f)$  where $M$ is a smooth manifold, $\xi$ is a bracket-generating  distribution and $f:\xi\to\RR$ is a fibrewise  Finsler norm. 

\begin{example}[(Sub-Finsler) Carnot groups]
Let $\gfrac$ be a (finite-dimensional) Lie algebra with subspace $\hfrac$ such that iterated Lie brackets of vectors in $\hfrac$ generate $\gfrac$. It follows that $\gfrac$ is nilpotent and can be endowed with a filtration:
\[ \gfrac_1 := \hfrac \subset \gfrac_2 := \hfrac \oplus [\hfrac,\hfrac] \subset \cdots. \]
The group law in the corresponding simply connected Lie group integrating it can be described explicitly, in terms of the structure constants of $\gfrac$, using the Campbell-Baker-Hausdorff formula. In particular, the underlying space of the group is the vector space $\gfrac$.

Given any Lie group $G$ with Lie algebra $\gfrac$, we can left-translate $\hfrac$ to yield a left-invariant distribution $\xi_\hfrac$; the same can be done for the subsequent entries in the filtration $\gfrac_i$. The associated Lie flag satisfies then $\Gamma^{(i)}(\xi_\hfrac) = \Gamma(\xi_{\gfrac_i})$. That is, the behavior of $\xi_\hfrac$ under Lie brackets is completely encoded in the pair $(\gfrac,\hfrac)$.

Lastly, we equip $\hfrac$ with a Finsler norm $f: \hfrac \to \RR$. Extending it by left-invariance to $\xi_\hfrac$ yields a sub-Finsler manifold $(G,\xi_\hfrac,f)$. These are commonly referred to as \textbf{Carnot groups}. \hfill$\triangle$
\end{example}

We will now discuss how the results from the previous Section may be adapted to this setting.

\subsection{Distributions as control systems}

Before we introduce costs, we can think of a distribution $(M,\xi)$ as a control system $\xi \subset TM$ in which the anchor is just the inclusion. If $\xi$ is bracket-generating, every attainable set is the whole manifold and the infimal time of arrival between any two points is zero (by increasing the speed of any given curve).

In the literature the following notation is commonly used:
\begin{definition}
Critical horizontal curves of $(M,\xi)$ are also called \textbf{singular}\footnote{They are called singular because they are critical points of the endpoint map and thus singularities of the space of admissible curves with given endpoints. We decided to use the word \emph{critical} instead in the general case in order to avoid confusion (particularly since we will look at critical curves for both $\xi$ and $(\xi,f)$).}.
\end{definition}

If we try to define the maximised Hamiltonian associated to $\xi$, we see that it should vanish along $\Ann(\xi)$ and be infinity everywhere else. This is equivalent to the fact that singular curves must have cotangent lifts contained in $\Ann(\xi)$. This can be further refined to prove:
\begin{proposition}[L. Hsu~\cite{H}]
Let $\gamma: [0,T] \to M$ be a horizontal curve in $(M,\xi)$. Then, $\gamma$ is singular if and only if there exists a lift $\lambda: [0,T] \to \Ann(\xi) \setminus M$ that is a characteristic of $\Ann(\xi)$. 
\end{proposition}
\begin{proof}[Sketch of proof]
As in Proposition \ref{prop:PSP}, we pushforward the system by the flow $\phi_t$ of the control $u_t$ generating $\gamma$. In this case, the image control system at the endpoint of the curve reads:
\[ E_t := (\phi_{t \to T})_*\xi_{\gamma(t)} \in T_{\gamma(T)}M. \]
The curve $\gamma$ is singular if and only if these spaces do not span $T_{\gamma(T)}M$. Identically, there should be a codirection $\lambda(T)$ supporting the span of the $E_t$. We write $\lambda(t) := \phi_{t \to T}^*\lambda(T)$.

Since $\gamma$ is an interval, we can fix a moving frame $X_1, \cdots, X_k$ of $\xi$ in a neighbourhood of $\gamma$. Then:
\[ d\phi_{t \to T}X_i(\gamma(t)) = X_i(\gamma(T)) - \int_T^t d\phi_{s \to T}[u_s,X_i](\gamma(s)) ds \]
which implies that $\lambda(t)$ must also annihilate each $[u_t,X_i](\gamma(t))$. Since both $u_t$ and $X_i$ take values in $\xi$, one may show that the following expression holds
\[ \lambda([u_t,X_i]) = -d\lambda(u_t,X_i), \]
for any extension of $\lambda(t)$ to a local $1$-form annihilating $\xi$. Identically, $\iota_{u_t}d\lambda|_\xi = 0$. 

Fix now a local coframe $\alpha_i$ of the annihilator $\Ann(\xi)$. Then, at the covector $\lambda$ (which we still extend to a local form annihilating $\xi$), the tautological $2$-form reads:
\[ \omega|_{\Ann(\xi)} =  \sum_i da_i \wedge \alpha_i + d\lambda, \]
where the $(a_i)$ are the fibre coordinates of $\lambda$ dual to the coframe. From this expression it immediately follows that any characterstic vector of $\omega|_{\Ann(\xi)}$ must be a lift of a vector in $\xi$. Furthermore, given a vector field $v$ tangent to $\xi$, its cotangent lift $\tilde v$ is given by the transport equation $\tilde v = \Lcal_v \lambda = \iota_v d\lambda$. We see that $\tilde v$ is tangent to the annihilator if and only if $\iota_v d\lambda|_\xi = 0$. Applying this reasoning to $u_t$ allows us to conclude.
\end{proof}

\subsection{Sub-Finsler costs}

Lastly, we focus on the main objects of this paper, horizontal curves with bounded speed:
\begin{definition}\label{def:sub-FinslerControl}
Given $(M,\xi,f)$ sub-Finsler, the associated \textbf{sub-Finsler control system} is given by the unit disc:
\[ \rho: \{v \in \xi \mid f(x) \leq 1\} \quad\to\quad TM. \]
\end{definition}
We will abuse notation and still refer by $(M,\xi,f)$ to the control system given by its unit disc. Observe that what we have effectively done is take the control system $(M,\xi)$ and refine it to a second control system $(M,\xi,f)$ that takes into account the cost $f$.

The unit disc is compact and convex and $\Acal_\infty(p)=M$ holds for $(M,\xi,f)$ due to the bracket-generating condition on $\xi$. As such, Proposition \ref{prop:Filippov} applies, proving that minimisers always exist between any two given points.

\begin{remark}
The sub-Finsler problem is extremely natural: Following Filippov, one should focus on control systems with convex and compact image. Then, one would study first the simplified situation in which the linear spaces spanned by the control have constant rank (i.e., they form a distribution), and the image of the control is a smooth subset. The remaining assumptions of sub-Finsler are that the zero vector is in the image of the control and that convexity is strict. \hfill $\triangle$
\end{remark}

\begin{remark} \label{rem:CarnotCaratheodory}
We defined the Carnot-Caratheodory distance $d$ to be the infimum of lengths of horizontal paths connecting two points. The previous discussion says that the topology defined by $d$ is the standard topology of $M$. However, assuming $\xi \neq TM$, $(M,d)$ is not bilipschitz equivalent to any Riemannian structure on $M$, cf.~\cite{P16}. This follows from the fact that its Hausdorff dimension in terms of $d$ is larger than its dimension as a manifold.

This can be intuitively seen as follows: In order to move in the direction $[X,Y] \notin \xi$, where $X,Y \in \Gamma(\xi)$, we have to construct an horizontal curve $\gamma$ that loops around the plane $\langle X, Y \rangle$ (effectively mimicking the Lie bracket geometrically). As such, displacement in $[X,Y]$ amounts to the signed area that $\gamma$ bounds in the $\langle X, Y \rangle$ plane. The latter is quadratic on the length of $\gamma$, suggesting that the sub-Riemannian distance along $[X,Y]$ should behave like a square root of the usual one. This will be apparent, for $3$-dimensional contact structures, from the discussion in Section \ref{sec:specialCase}. \hfill$\triangle$
\end{remark}

\subsection{Cotangent viewpoint} \label{sssec:cotangentsub-Finsler}

We now particularise the cotangent discussion to the sub-Finsler setting. Recall that the dual norm of $f$ is given by the expression:
\begin{align*}
H_f: T^*M \quad\to\quad& \RR \\
H_f(\lambda) \quad := \quad & \max_{v \in \xi, f(v) \leq 1} \lambda(v).
\end{align*}
Note that $H_f$ factors through the projection $T^*M \to \xi^*$ and, as such, it is degenerate and not a finsler norm. Indeed, $H_f$ is invariant under translations along the annihilator $\Ann(\xi)$ and is strictly convex on any of its complements. We can summarise the situation as follows:
\begin{lemma} \label{lem:maximisedHamiltonian}
Given a sub-Finsler control system $(M,\xi,f)$:
\begin{itemize}
\item Its maximised Hamiltonian is $H_f$. It is smooth away from $H_f^{-1}(0) = \Ann(\xi)$.
\item A curve is critical if and only if it has a cotangent lift that is a characteristic of $H_f^{-1}(c)$, $c \geq 0$.
\item A minimising curve is abnormal if and only if it is singular for $\xi$, i.e. it has a lift that is a characteristic of $\Ann(\xi)$.
\end{itemize}
\end{lemma}

In particular, the characteristic curves contained in $H_f^{-1}(0)$ do not depend on $f$, just on $\xi$. Still, a singular curve of $(M,\xi)$ may fail to be minimising and thus fail to be abnormal for $(M,\xi,f)$. A lot of the research in sub-Riemannian Geometry has to do with detecting abnormals.

In this paper we will embrace the Hamiltonian viewpoint, disregard abnormals altogether, and thus work solely with the characteristics of $H_f^{-1}(c)$, $c > 0$. Homogeneity says that we can focus on the unit cylinder:
\begin{definition}\label{def:cogeodesicFlow}
We will refer to the Hamiltonian flow of $H_f$ at the level set $H_f = 1$ as the \textbf{sub-Finsler (co)geodesic flow} of $(M,\xi,f)$.
\end{definition}

Using the fact that the level sets of $H_f$ are convex in any complement of $\Ann(\xi)$ we deduce:
\begin{lemma} \label{lem:correspondingVector}
Let $\lambda$ be a Hamiltonian flowline in $\{H_f = 1\}$. Then $(\pi \circ \lambda)'(t)$ is the unique vector $v \in \xi$ maximising the evaluation $\lambda(t)(v)$.
\end{lemma}

\subsubsection{Annihilator-invariant Hamiltonians}

Discarding strict abnormals motivates us to define:
\begin{definition}
Let $(M,\xi)$ be a bracket-generating distribution. A smooth function $H: T^*M \to \RR$ is a \textbf{$\xi$-Hamiltonian} if it is fibrewise invariant under translations along $\Ann(\xi)$.
\end{definition}
In particular, their level sets are non-compact contact manifolds whenever $\xi \neq TM$ (do note that they are transverse to the Liouville vector field). A remarkable open question is whether a Floer-style theory can be developed for this class of Hamiltonians, effectively bringing Symplectic Topology techniques into sub-Riemannian Geometry.

\section{Definition of the billiard flow} \label{sec:billiardFlow}

Up to this point we have been working with $(M,\xi,f)$ complete without boundary. Our goal in this Section is to introduce the sub-Finsler billiard problem, both from a symplectic and a variational perspectives. We will then show that the two agree with one another. This correspondence is already well-known for Riemannian, Finsler~\cite{AKO}, and Lorentz billiards~\cite{KT}, where it has deep consequences.

\subsection{The table}

We consider a closed subset $U\subseteq M$ with smooth boundary $\partial U$; this will be our billiard table. We denote $\xi_\partial := T\partial U \cap \xi$, which we call the \textbf{boundary distribution}\footnote{In 3-dimensional Contact Topology literature this is often called the \textbf{characteristic foliation}. We have opted to avoid this terminology to minimise confusion.}.

Since we do not assume transversality of $\partial U$ with respect to $\xi$, $\xi_\partial$ may have singularities. We will discuss what this entails for the reflection law in Subsection \ref{sssec:degenerateReflections}. We denote their complement by $\partial U^\circ$; in this region, $f$ restricts to a finsler norm on $\xi_\partial$. We note that $\xi_\partial$ may not be bracket-generating and, as such, the triple $(\partial U^\circ,\xi_\partial,f)$ may not be a sub-Finsler manifold. However, one can apply the exact same reasoning we used in the previous Section and produce an associated (co)geodesic flow; see Subsection \ref{sssec:boundaryGeodesics}. For instance, if $\xi_\partial$ is involutive, this will define the usual leafwise cogeodesic flow.

It is a well-known theme in Billiards that geodesics of the boundary may appear in the compactification of the space of billiard trajectories. We will explore this in Subsection \ref{sssec:glide}.

\subsection{The symplectic viewpoint} \label{ssec:symplecticViewpoint}

Let $H_f: T^*M \to \RR$ be the maximised Hamiltonian associated to $(M,\xi,f)$. We write $S_f := H_f^{-1}(1)$ for the unit cotangent cylinder associated to $H_f$. Similarly, we write $D_f := H_f^{-1}((-\infty,1])$ for its interior.

As explained in Subsection \ref{sssec:cotangentsub-Finsler}, once we disregard strict abnormals, sub-Finsler geodesics in the interior of the table correspond to projections of characteristic flow lines of $S_f|_U$. It is natural to extend this symplectic definition of the dynamics to encode the reflection law as well:
\begin{definition}
Consider the sub-Finsler billiard table $(U,\xi,f)$. Its \textbf{microlocal realisation} is the piecewise manifold $B_f := S_f|_U \cup D_f|_{\partial U} \subset T^*M$, endowed with the characteristic dynamics of each piece.
\end{definition}
We recall that characteristic curves in $S_f|_U$ have a preferred parametrisation as Hamiltonian orbits of $H_f$; this corresponds to the parametrisation by arclength of their projections. Similarly, a function $G: M \to \RR$ with $\partial U$ as a regular level set may be lifted to $T^*M$ to produce a preferred Hamiltonian parametrisation of the characteristics of $D_f|_{\partial U}$; one needs to choose $G$ with an appropriate orientation, see Remark \ref{rem:choiceG} below. Thus:
\begin{definition}\label{def:billiardCharacteristic}
Let $I \subset \RR$ be an interval, possibly infinite. A \textbf{billiard characteristic} $\lambda: I \to B_f$ is a Lipschitz characteristic curve with $\dot\lambda$ equal to the Hamiltonian vector fields $X_{H_f}$ or $X_G$ almost everywhere.
\end{definition}

We can then define:
\begin{definition}\label{def:billiadTrajectory}
A curve $\gamma: I \to U$ is an \textbf{unreduced billiard trajectory} if it is the projection of a billiard characteristic.

A curve $\gamma: I \to U$ is a \textbf{(reduced) billiard trajectory} if it is a reparametrisation by arclength of an unreduced billiard trajectory.
\end{definition}

Since $S_f$ is the unit cylinder, it follows that:
 \begin{lemma}
Let $\lambda: [a,b] \to B_f$ be a billiard characteristic. Then:
\[ \len(\pi \circ \lambda) = \int_a^b \lambda^* \lambda_\std, \]
where $\lambda_\std$ denotes the standard Liouville 1-form in cotangent space. That is, the action of $\lambda$ is the length of $\pi \circ \lambda: [a,b] \to U$ (which is also the length of the corresponding reduced billiard trajectory).
\end{lemma}

\subsubsection{Reflections}

The characteristic flow in the piece $D_f|_{\partial U}$ is the one responsible for the reflection law. We readily compute:
\begin{lemma} \label{lem:sympReflectionLaw}
The characteristic foliation of $D_f|_{\partial U}$ is given by the lines parallel to $\Ann(T\partial U)$.
\end{lemma}
\begin{proof}
The claim follows from the fact that the Hamiltonian flow of $G$ is simply fibrewise translation along $dG$, which annihilates $T\partial U$.
\end{proof}

We now introduce some notation. Let $\lambda: [a,b] \to D_f|_{\partial U}$ be a billiard characteristic with $\lambda(a) \neq \lambda(b) \in S_f|_{\partial U}$. Suppose that we can continue $\lambda$ forwards (from $a$) and backwards (from $b$) in time by following the sub-Finsler cogeodesic flow on $U$ for a short time.

\begin{definition}
We say that:
\begin{itemize}
    \item the \textbf{reflection law} $\lambda(a) - \lambda(b) \in \Ann(T\partial U)$ holds.
    \item $\lambda|_{[a,b]}$ is a \textbf{reflection} with \textbf{ingoing momentum} $\lambda(a)$ and \textbf{outgoing momentum} $\lambda(b)$. 
\end{itemize}
\end{definition}
As one would expect, reflections have zero action and project to constant curves on $\partial U$.

\subsubsection{Degeneracy of reflections} \label{sssec:degenerateReflections}

Suppose $\lambda: (-1,0] \to B_f$ is a billiard characteristic with $\pi \circ \lambda(0) \in \partial U$  and $\pi \circ \lambda|_{(-1,0)}$ a characteristic of $S_f$. We want to determine the outgoing momentum $\lambda_{\rm out}$ from the ingoing momentum $\lambda(0)$.

It may be the case that the line $L$ through $\lambda(0)$ parallel to $\Ann(T\partial U)$  is transverse to $S_f|_{\partial U}$. Therefore, there exists a unique momentum $\lambda_{\rm out}$, other than $\lambda(0)$, contained in $L \cap S_f|_{\partial U}$. The reflection segment in-between the two may be concatenated with $\lambda$, and then we can append the characteristic of $S_f$ starting at $\lambda_{\rm out}$. We say that this is a \textbf{non-degenerate reflection}.

When this reasoning does not apply, the reflection dynamics may be ill-defined; we call these \textbf{degenerate reflections}. Let us break down what may happen.

The first possibility is that $L$ is contained in $S_f$. This is the case if and only if $\xi\subseteq T\partial U$, i.e., at singular points of the boundary distribution $\xi_\partial$. In this case, the reflection law is not defined, as there is no well--defined outgoing momentum. We discuss, in a concrete example, the limit behavior of the billiard flow near a singular point in Subsection \ref{sssec:critical}.

The second possibility is that $L$ intersects $S_f$ only at $\lambda(0)$, i.e., it has an outer tangency with $D_f$. Then we distinguish two cases: Maybe $\pi \circ \lambda$ can be continued beyond $0$ as a sub-Finsler geodesic of $U$. Then we ignore the reflection point, which becomes an inner tangency of the billiard trajectory with the boundary.

The alternative is that the curve cannot be continued as a geodesic, but it may continue as a \emph{glide orbit}. In the classical setting, these are curves tangent to $\partial U$ that arise as limits of sequences of trajectories with progressively more reflections of decreasing angle; the story in the sub-Finsler case is more involved. The treatment of glide orbits is important, since they provide a completion of the space of billiard curves. We discuss them in some detail in Subsection \ref{sssec:glide}.

\subsubsection{The space of sub-Finsler geodesics}

Degeneracies of the reflection law can be best understood using the semilocal structure of the space of geodesics along $\partial U$.

Recall that, already in the Riemannian setting, the space of geodesics (with any natural topology) is in general extremely complicated due to long term dynamical behavior; for instance, it is often not Hausdorff. Despite of this, the Hamiltonian description of the cogeodesic flow (also in our sub-Finsler setting), implies that locally, the space of (non strictly-abnormal) geodesics can be obtained by symplectic reduction from the unit cylinder $S_f \subset T^*M$.

Given the hypersurface $\partial U \subset M$, we can study the geodesics in a germ of neighbourhood $\Op(\partial U) \supset \partial U$. Let $\gamma$ be a germ of geodesic with $\gamma(0) \in  \partial U$ a transverse (and thus unique) intersection. If we denote by $\lambda$ the corresponding momentum germ, we can identify $\gamma'(0)$ with $\lambda(0)$ using Lemma \ref{lem:correspondingVector}. With this in mind we introduce the notation:
\begin{definition}\label{def:outwardInward}
A covector $p \in (T^*M \setminus \Ann(\xi))|_{\partial U}$ is:
\begin{itemize}
    \item \textbf{tangent} if the unique vector $v \in \xi$ maximising $p(v)$ is tangent to $\partial U$.
    \item \textbf{(strictly) outward pointing} if the unique vector $v \in \xi$ maximising $p(v)$ is (strictly) outward pointing.
    \item \textbf{(strictly) inward pointing} if the unique vector $v \in \xi$ maximising $p(v)$ is (strictly) inward pointing.
\end{itemize}
We write $S^0_f|_{\partial U}$, $S^+_f|_{\partial U}$, and $S^-_f|_{\partial U}$ for the subspaces of covectors in $S_f|_{\partial U}$ that are tangent, strictly outward, and strictly inward, respectively.
\end{definition}

Applying reduction we deduce:
\begin{proposition}\label{prop:ReflectionIsSymplectomorphism}
The following statements hold:
\begin{itemize}
    \item $S^\pm_f|_{\partial U}$ is a smooth fibre bundle over $\partial U^\circ$, the set of regular points of the boundary distribution.
    \item $S^\pm_f|_{\partial U}$ is a symplectic submanifold of $(T^*M, \omega_\std)$.
    \item The space of outward pointing geodesics along $\partial U$ is symplectomorphic to $S^+_f|_{\partial U}$.
    \item The space of inward pointing geodesics along $\partial U$ is symplectomorphic to $S^-_f|_{\partial U}$.
    \item The reflection law from Lemma~\ref{lem:sympReflectionLaw} defines a symplectomorphism between $S^+_f|_{\partial U}$ and $S^-_f|_{\partial U}$.
\end{itemize}
\end{proposition}
\begin{proof}
The first claim is immediate. For the second and fifth claims we recall that the tautological symplectic form $\omega_\std$ admits the following splitting along $S^+_f \partial U$:
\[ \omega_\std|_{TT^*M|_{S^+_f|_{\partial U}}} = \omega_\std|_{T^*S^+_f|_{\partial U}} \oplus \omega_\std|_{\Ann(\partial U) \oplus \langle X_{H_f}\rangle}. \]
Indeed, $\Ann(\partial U)$ is symplectically orthogonal to $S^+_f \partial U$ because the latter fibres over $\partial U$. Similarly, $X_{H_f}$ is linearly independent from $\Ann\partial U$ and annihilates $S^+_f \partial U$ because the latter is part of a level set of $H_f$. The claims then follow because the reflection law is precisely given by displacement along $\Ann(T\partial U)$. The third and fourth claims follow from the discussion above.
\end{proof}
The Proposition describes only non-degenerate reflections, pointing out that the reflection law may be ill-defined when transversality fails.

\begin{remark}\label{rem:choiceG}
The preferred function $G$,  which has $\partial U$ as a regular level set, should satisfy that the flow of $X_G$ goes from $S^+_f|_{\partial U}$ to $S^-_f|_{\partial U}$, realising the symplectomorphism between the two. I.e. the reflection takes outward pointing covectors to inward pointing ones. This yields a sign constraint for $dG$. \hfill$\triangle$
\end{remark}

\subsubsection{A cotangent look at the boundary} \label{sssec:boundaryGeodesics}

Continuing with the language introduced in the previous Subsection, and in order to set up some notation, we now look at the subspace of tangent vectors. Observe first that $S^0_f|_{\partial U}$ is not a smooth bundle, since its fibres change dimension at the singularities of $\xi_\partial$. We therefore focus on its restriction to the regular part of the boundary distribution:
\begin{lemma}
The subspace $S^0_f|_{\partial U^\circ}$ is a smooth fibre bundle over $\partial U^\circ$. It is contactomorphic to the unit cylinder of the restricted structure $(\partial U^\circ,\xi_\partial,f)$.
\end{lemma}
\begin{proof}
The claim is tautological: the identification between the two follows by regarding the tangent covectors in $S^0_f|_{\partial U^\circ}$ as covectors in the boundary.
\end{proof}
In general, $X_f$ will not be tangent to $S^0_f|_{\partial U^\circ}$. Identically, the geodesic flow in $U$ is not necessarily tangent to $\partial U$. However, just like in Riemannian geometry, the boundary geodesic flow differs from the ambient geodesic flow by a projection (along the normal acceleration) making the latter tangent:
\begin{lemma} \label{lem:normalCurvature}
There is a unique vector field $\tilde X_f$ in $S^0_f|_{\partial U^\circ}$ satisfying 
\[ \tilde X_f = X_f \quad  \rm{mod}\,\Ann(T\partial U). \]
Additionally, under the identification provided by the previous Lemma, $\tilde X_f$ is the Hamiltonian vector field, along the unit level set, of the restriction of $H_f$ to the cotangent bundle of $U^\circ$.
\end{lemma}
\begin{proof}
At points in $S^0_f|_{\partial U^\circ}$, the projection of $X_f$ to the base manifold is, by definition, tangent to $\partial U$. Again by definition, $X_f$ is tangent to $S_f$. Then, the normal bundle of $S^0_f|_{\partial U^\circ}$ within $S_f|_{\partial U^\circ}$ is spanned by $\Ann(T\partial U)$, proving the first claim. The second claim follows by recalling that the cotangent bundle of $\partial U$ is obtained from $T^*M$ by reduction of a regular level set of $G$. 
\end{proof} 

\subsubsection{Gliding and creeping orbits} \label{sssec:glide}

We now tackle the phenomenon of gliding. We think of these as elements in some compactification of the space of billiard trajectories so, more generally, we define:
\begin{definition}
A curve $\gamma: I \to U$ is a \textbf{limit trajectory} if it is a $C^0_{loc}$-limit of (reduced) billiard trajectories. If $\gamma(I)\subseteq\partial U$, then we call it a \textbf{gliding trajectory}.
\end{definition}

We first observe:
\begin{lemma}\label{lem:GlidingAreAdmissible}
The following statements hold:
\begin{itemize}
    \item Limit trajectories are Lipschitz and tangent to the distribution $\xi$ almost everywhere.
    \item Gliding trajectories are tangent to the boundary distribution $\xi_\partial$ almost everywhere.
    \item If $U$ is compact, any sequence of billiard trajectories has a $C^0_{loc}$--convergent subsequence.
\end{itemize}

\end{lemma}
\begin{proof}
The last claim follows from Lemma \ref{lem:LipschitzLimit}. The Lemma also says that any $C^0_{loc}$-convergent sequence of admissible orbits of $\xi$ converges to a Lipschitz trajectory still tangent to $\xi$. Similarly, gliding trajectories must have velocity in $T\partial U\cap\xi=\xi_\partial$.
\end{proof}

\begin{remark} \label{rem:convergingToConstant}
Length is lower semi-continuous in $C^0_{loc}$-topology; this has to do with the fact that the unit ball is the convex hull of the unit sphere. In fact, any nontrivial sub-Riemannian manifold has sequences of geodesics (defined over $\RR$) converging in $C^0_{loc}$ to constant curves. We already hinted at this in Remark \ref{rem:CarnotCaratheodory}: The reader can picture a sequence of curves looping progressively faster in $\xi$ in order to travel transversely to it. Since the transverse displacement corresponds roughly to an area bounded in $\xi$, it has a different scale than the length. As we take the limit, the ratio of vertical displacement to length goes to zero. More explicitly, such a sequence can be produced by fixing a sequence of starting momenta $\lambda_i(0)$ diverging fibrewise in the unit cylinder; a concrete example is given in Remark \ref{rem:convergingToConstantExample}. \hfill$\triangle$
\end{remark}

We are interested in proving the following folklore theorem in Billiards: Gliding trajectories follow the geodesic flow in the boundary. The main problem when tackling this result is that the number of reflections increases as we take a sequence of trajectories converging to the boundary; this needs to be controlled quantitatively. This may be approached in two different ways: We may embrace the classical interpretation of billiards and think of them as piecewise trajectories; then the goal is showing that the change in direction at the reflections corresponds in the limit to the curvature of the boundary. Alternatively, we follow the microlocal viewpoint and work with sequences of cotangent lifts. These lifts are parametrised differently than their projections, since the reflection segments following $X_G$ contribute to the parametrisation. This is especially problematic as we take limits, since reflection points will become dense.

We can strengthen our definition of convergence to avoid these problems:
\begin{definition}
    A limit trajectory $\gamma: I\to U$ is a \textbf{$C^1$-limit trajectory} if, for some $C^0$-approximating sequence $\gamma_i$, the corresponding billiard characteristics $\lambda_i$ converge in $C^0_{loc}$ to a cotangent lift $\lambda$ of $\gamma$. If $\gamma(I)\subseteq\partial U$, then we call it a \textbf{$C^1$-gliding trajectory}.
\end{definition}
\begin{remark}
The curves $\gamma_i$ are piecewise smooth, and each smooth segment lifts to a segment of the corresponding $\lambda_i$. According to Lemma \ref{lem:correspondingVector}, the latter encodes the velocity of the former. In particular, $\lambda_i$ may be considered as an incarnation of the first jet of $\gamma_i$ in which the first derivative is made continuous by adding the reflection segments. It is for this reason that we call these $C^1$-limits. \hfill$\triangle$
\end{remark}

We note:
\begin{lemma}
$C^1$-limit trajectories disjoint from $\partial U$ are billiard trajectories. In particular, they are smooth.
\end{lemma}
\begin{proof}
Given a $C^1$-limit disjoint from the boundary, we deduce that it is a limit of billiard trajectories disjoint from the boundary. As such, these are simply normal geodesics and thus projections of Hamiltonian orbits of $X_f$. A limit of Hamiltonian orbits is a Hamiltonian orbit, proving the claim.
\end{proof}

We then show that a $C^1$-converging sequence of billiard characteristics has reflection angles going to zero. Recall from Definition~\ref{def:outwardInward} that $S^0_f|_{\partial U}$ is the space of covectors tangent to $\partial U$. Even though it plays no role in the proof, be aware that it fails to be smooth at the singularities of $\xi_\partial$.
\begin{lemma} \label{lem:reflectionsGoToZero}
The limit characteristic lift of a non-constant $C^1$-gliding orbit takes values in $S_f^0|_{\partial U}$.
\end{lemma}
\begin{proof}
    Let $\lambda_i: I\to B_f$ be a sequence of billiard characteristics converging in $C^0_{loc}$ to a curve $\lambda: I\to B_f|_{\partial U}$. The projection $\gamma:I\to \partial U$ of $\lambda$ is an (unreduced) $C^1$-gliding trajectory, which is the $C^0_{loc}$-limit of the projections $\gamma_i := \pi \circ \lambda_i$. Denote by $\bar\gamma$ and $\bar\gamma_i$ the corresponding reduced trajectories. By Lemma~\ref{lem:LipschitzLimit}, $\bar\gamma$ is an admissible curve of $(\partial U,\xi_\partial, f)$.
    
    Recall that the microlocal realisation consists of the piecewise smooth parts $B_f = S_f\cup D_f|_{\partial U}$. In this piecewise space, we define a control system $\Ccal \subset TB_f$ consisting of the vector fields $X_f$ on $S_f$, $X_G$ on $D_f|_{\partial U}$, and $\textrm{ConvexHull}(X_f,X_G)$ on the intersection $S_f|_{\partial U}$. Each of the $\lambda_i$ is an admissible curve of $\Ccal$. Even though $B_f$ is piecewise, we can apply Lemma~\ref{lem:LipschitzLimit} to it by working locally in $T^*M$. We deduce that the limit curve $\lambda$ is an admissible curve of $\Ccal$.
    
    Since $\gamma$ is contained in $\partial U$, we must have $\lambda(t)\in D_f|_{\partial U}$ for all $t \in I$. Suppose for the sake of absurdity that $\lambda(t_0) \notin S_f^0|_{\partial U}$ for some $t_0$. Then, for some maximal positive interval $t_0 \in I_0 \subsetneq I$, the curve $\lambda(I_0)$ is a characteristic line of $D_f|_{\partial U}$. Unless the endpoints of $I_0$ are in $S_f^0|_{\partial U}$, $\lambda$ will escape $\partial U$. However, $S_f^0|_{\partial U}$ is precisely the set of points in $D_f|_{\partial U}$ in which the corresponding characteristic line has an outer tangency, yielding a contradiction.
\end{proof}
We refer the reader to \cite{AO} for the  analogous statement in the Finsler setting. Finally, we prove the desired folklore theorem for $C^1$-gliding trajectories. 
\begin{proposition}\label{prop:boundaryGeodesic}
    A $C^1$-gliding trajectory is a (possibly singular) reparametrisation of a boundary geodesic.
\end{proposition}
\begin{proof}
We use the same setup and notation as in the previous Lemma \ref{lem:reflectionsGoToZero}. Our reasoning showed that  $\lambda(t) \in S_f^0|_{\partial U}$ for all $t \in I$ and thus $\dot\lambda\in \textrm{ConvexHull}(X_f,X_G)$. Identically, $\dot\lambda(t)$ may be written as a linear combination $a(t)X_f + (1-a(t))X_G$ almost everywhere. Do note that $a$ is non-zero almost everywhere.

Since $d\pi(X_G)=0$, we have that 
\[ \dot\gamma = d\pi(\dot\lambda(t)) = a(t) d\pi(X_f(\lambda(t))) \in T\partial U. \]
This means, by definition of $X_f$, that $\lambda(t)$ attains its maximum, over the spaces $\{v\in\xi|f(v)\leq 1\}$ and $\{v\in\xi_\partial|f(v)\leq 1\}$, at $\frac 1{a(t)}\dot\gamma \in \xi_\partial$. Here, we realise that $a(t) = f(\dot\gamma(t))$. Reparametrising $\gamma$ by unit speed, we obtain a curve $\tilde\gamma$ in $(\partial U, \xi_\partial, f)$ such that $\lambda$ is the $f$-dual of $\dot\gamma$.

According to Lemma \ref{lem:normalCurvature}, away from the singularities of $\xi_\partial$, there is a single value $a(t)$ such that the combination $a(t)X_f + (1-a(t))X_G$ is in fact tangent to $S_f^0|_{\partial U}$. We can interpret the component $(1-a(t))X_G$ as the normal acceleration that holds the curve on $\partial U$. This precisely makes the curve follow the boundary geodesic flow.
 \end{proof}

\begin{remark}
This proof applies as well to other billiard settings, like the Riemannian billiard or Finsler billiard, as they are special cases of the problem we study (with $\xi=TM$). \hfill $\triangle$
\end{remark}

In Subsection \ref{subsubsec:creeping} we will analyse the gliding phenomenon in detail on a particularly symmetric example. We will make the curious observation that, in the examined examples, gliding orbits and $C^1$-gliding orbits agree up to reparametrisation.

\subsubsection{Creeping orbits}

In Subsection \ref{subsubsec:creeping} we will also look at instances of:
 \begin{definition}
 A curve $\gamma: \RR \to \partial U$ is a \textbf{creeping trajectory} if there is a sequence $\gamma_i$ of billiard trajectories and real numbers $a_i$ such that $\gamma(t)$ is the $C^0$-limit of $\lim_{i \to \infty} \gamma_i(a_i t)$.
 \end{definition}
 The class of creeping orbits is sub-Finsler in nature and displays interesting behaviors that have no classical analogue. In Remarks \ref{rem:convergingToConstantExample} and \ref{rem:diffOperators} we discuss the fact that, in the contact case, Reeb orbits are creeping orbits.

\subsection{The variational approach} \label{ssec:variationalViewpoint}

We now study the reflection law from a variational/control theoretical viewpoint. Some of this discussion is not particular to the sub-Finsler setting, and applies to more general control systems $\rho: C \to TM$. In Appendix~\ref{appendix:nonAutonomous} we extend this discussion to non-autonomous sub-Finsler billiards.

In the following, a domain with smooth boundary $U \subset M$ still serves as the table.

\subsubsection{Diamond control systems}

Let $q_1,q_2 \in U \setminus \partial U$ be two points in the interior of the table. A once reflected curve from $q_1$ to $q_2$ should (infinitesimally) minimise the length among admissible curves that go from $q_1$ to $\partial U$ and then to $q_2$. A subtle point is that the time of arrival at the boundary is an unknown. The way to proceed is to define a new control problem addressing this.

\begin{definition}\label{def:diamond}
We define the \textbf{diamond control system} associated to $\rho: C \to TM$ to be 
\[ \rho_{M \times M}: C_{M \times M} \to TM \times TM \]
where
\begin{itemize}
\item $C_{M \times M} := C \times C \times [0,1]$,
\item $\rho_{M \times M}(u_1,u_2,s) := (s\rho(u_1),-(1-s)\rho(u_2))$.
\end{itemize}
Its image can be expressed as follows:
\[ \{(s v, -(1-s)w) \in TM \times TM \mid s \in [0,1], v,w \in \rho(C) \}. \]
\end{definition}

\begin{remark}
The diamond of a strictly convex system is convex but not strictly convex. In particular, the diamond of a sub-Finsler system is not sub-Finsler.\hfill$\triangle$
\end{remark}

Consider $\gamma: [0,t_1+t_2] \to M$, an admissible curve from $\gamma(0) = q_1$ to $\gamma(t_1+t_2) = q_2$ with $\gamma(t_1) \in \partial M$ the only point in the boundary. Passing to the diamond control system, we can define a curve $\gamma_{M \times M}: [0,t_1+t_2] \to M \times M$:
\[ \gamma_{M \times M}(t) := \left(\gamma|_{[0,t_1]}\left(t\dfrac{t_1}{t_1+t_2}\right),\gamma|_{[t_1,t_1+t_2]}\left((t_1+t_2-t)\dfrac{t_2}{t_1+t_2}\right)\right) \]
connecting the point $(q_1,q_2)$ with the diagonal of the boundary $\Delta_{\partial M} := \{(q,q)\mid q\in\partial M \}$. By construction, $\gamma_{M \times M}$ is admissible for the diamond system (this is why the control in the second factor of the diamond is reversed) with $s\equiv \frac{t_1}{t_1+t_2}$. We remark:

\begin{lemma}\label{rectangle}
$\gamma_{M \times M}$ is part of an infinite dimensional family $\Fcal_\gamma$ of admissible curves with fixed endpoints, all of which have the same length and the same projections to each factor (up to reparametrisation).
\end{lemma}
\begin{proof}
Consider the map with rectangular domain:
\begin{align*}
\Gamma: R = [0,t_1] \times [0,t_2] \quad\to\quad & M \times M \\
 (\tau_1,\tau_2) \quad\to\quad& (\gamma(\tau_1),\gamma(t_1+t_2-\tau_2))
\end{align*}

From $\Gamma$ we define elements in $\Fcal_\gamma$ by choosing a time function $\tau:[0,t_1+t_2]\to R$ that is monotone in each coordinate, has unit speed with respect to the norm $|d\tau_1|+|d\tau_2|$, and connects $(0,0)$ to $(t_1,t_2)$. Indeed, all curves $\nu \in \Fcal_\gamma$ are described as $\Gamma\circ\tau$ for some such $\tau$ and from construction they have the same length. In particular, $\gamma_{M\times M}$ is parametrised by the diagonal of $R$. It is clear that $\gamma$ can be uniquely recovered from any of them by projection and reparametrisation.
\end{proof}

Minimality is preserved when we pass to the diamond reformulation:
\begin{lemma}\label{lem:diamondMinimizer}
The following two conditions are equivalent:
\begin{itemize}
\item $\gamma$ is a minimiser of $\rho: C \to TM$ among curves from $q_1$ to $q_2$ and touching the boundary of $U$ (at an indeterminate time).
\item $\gamma_{M \times M}$, or any of the other curves in $\Fcal_\gamma$, is a minimiser of the diamond system among curves from $(q_1,q_2)$ to $\Delta_{\partial M}$.
\end{itemize}
\end{lemma}
\begin{proof}
First note that one can recover a curve $\gamma$ from the corresponding $\gamma_{M \times M}$ by projection. By construction, one is admissible if and only if the other one is admissible. Additionally both are parametrised by the same interval. This concludes the claim.
\end{proof}

This motivates the following definition:
\begin{definition}\label{def:variationalBilliardTrajectory}
A \textbf{variational billiard trajectory} is an admissible curve $\gamma:I\to U$ satisfying:
\begin{itemize}
    \item It is piecewise smooth, and its points of non-smoothness form a discrete set $A \subset I$.
    \item In the complement of $A$, $\gamma$ is a minimiser (possibly only locally in time).
    \item given $t \in A$, it holds that $\gamma(t) \in \partial U$ and there is an $\varepsilon>0$ such that $\gamma|_{[t-\varepsilon,t+\varepsilon]}$ satisfies the conditions in Lemma~\ref{lem:diamondMinimizer}.
\end{itemize}
\end{definition}

\subsubsection{Reflection laws for control systems}

Following the prior discussion, we should look at control problems in which the endpoints are allowed to vary within a submanifold. We recall:
\begin{proposition} \label{prop:PMPboundary}
Let $\eta: D \to TW$ be a control system. Let $V_1,V_2 \subset W$ be smooth submanifolds. Suppose $\nu: [0,T] \to W$ minimises the arrival time between $V_1$ and $V_2$. Then, there exists a cotangent lift $\lambda: [0,T] \to W$ satisfying PMP, $\dot\lambda(0) \in \Ann(V_1)$, and $\dot\lambda(T) \in \Ann(V_2)$.
\end{proposition}

We will apply this to our diamond system, where $V_1=\{(q_1,q_2)\}$ and $V_2=\Delta_{\partial U}$. Note that
\[ \Ann T\Delta_{\partial U} = \{(\lambda_1,\lambda_2) \mid \lambda_1+\lambda_2 \in \Ann T\partial U\}. \]

\begin{proposition}\label{prop:reflectionLaw}
Let $\gamma: [0,t_1+t_2] \to U$ be admissible, connecting $\gamma(0) = q_1$ to $\gamma(t_1+t_2) = q_2$, and $\gamma(t_1) \in \partial U$ a point in the boundary.

Then, $\gamma$ is a minimiser among such curves if and only if there is a cotangent lift $\lambda: [0,t_1+t_2] \to T^*U$ such that the pieces $\lambda|_{[0,t_1]}$ and $\lambda|_{[t_1,t_1+t_2]}$ satisfy PMP and such that
\begin{equation} \label{eq:reflectionLaw}
\lambda(t_1^-) - \lambda(t_1^+) \in \Ann(\partial U),
\end{equation}
where $t_1^-$ denotes the limit as we approach $t_1$ from the left and $t_1^+$ is the limit from the right.
\end{proposition}
\begin{proof}
The curve $\gamma$ is a minimiser if and only if $\gamma_{M \times M}$ is a minimiser. This is equivalent to being able to find a cotangent lift $\lambda_{M \times M}$ satisfying PMP with $\dot\lambda_{M \times M} \in \Ann(\Delta_{\partial U})$.

The curve $\lambda_{M \times M}$ defines a cotangent lift $\lambda_1$ of $\gamma|_{[0,t_1]}$ (resp.\ a lift $\lambda_2$ of $\gamma|_{[t_1,t_1+t_22]}(t_1+t_2-t)$) by evaluating on vectors parallel to the first (resp. second) factor and reparametrising suitably. These satisfy PMP only if $\lambda_{M \times M}$ does. Lastly, the endpoint condition for $\lambda_{M \times M}$ implies that $\lambda_1(t_1) - \lambda_2(t_2)$ annihilates the boundary. Note that the minus sign comes from the fact that $\lambda_{M\times M}$ reverses the second path.
\end{proof}

We say that Equation~(\ref{eq:reflectionLaw}) is the (control theoretical) \textbf{reflection law} that variational billiard trajectories must satisfy. As we saw already using the symplectic formalism, for a general control system, the reflection law may not uniquely determine the outgoing momentum $\lambda(t_1^+)$ from the ingoing momentum $\lambda(t_1^-)$.

\subsubsection{The reflection law in the sub-Finsler setting}

We particularise the discussion to the sub-Finsler setting:
\begin{lemma}
Let $(M,\xi,f)$ be a sub-Finsler manifold. Then, the associated diamond control system is given by the unit disc of the system
\[ (M \times M, \xi \oplus \xi, f \oplus \widetilde f), \]
where $\widetilde f(w) = f(-w)$.
\end{lemma}
It is worth pointing out that the unit sphere of $f \oplus \bar f$ in $\xi\oplus\xi$ is not a smooth sphere, but the boundary of a diamond. As such, it has corners, which are precisely the unit spheres in each factor. 
\begin{proposition}\label{prop:sympIsVariational}
The symplectic and control theoretical reflection laws are equivalent for sub-Finsler billiards. This means that reduced billiard trajectories are exactly the variational billiard trajectories.
\end{proposition}
\begin{proof}
In this setting, Equation~(\ref{eq:reflectionLaw}) coincides with Lemma~\ref{lem:sympReflectionLaw}, but this does not completely determine a reflection law. What is left to show is that the reflection law defined by the diamond control system preserves the dual norm:
\[H_f(\lambda(t_1^+))=H_f(\lambda(t_1^-)),\]
which in the description of the microlocal realisation was achieved automatically by construction. If this is true, then the ingoing and outgoing momentum lie on the same energy surface of $H_f$, which leads to the desired equivalence.

First, we use Lemma~\ref{rectangle} to choose a minimising curve whose terminal velocity does not vanish in either component, which is equivalent to $\tau_1'(T)\neq 0 \neq \tau_2'(T)$. This means that the velocity of arrival \[\gamma_{M\times M}'(T)=(\tau_1'\gamma_1'(t_1),-\tau_2'\gamma_2'(t_2))\]
lies in a non-singular part of the image of the diamond, c.f.\ Definition~\ref{def:diamond}.
Let $\lambda_{M\times M}=(\lambda_1,\lambda_2)$ be a cotangent lift of $\gamma_{M\times M}$ satisfying PMP. We know that $\lambda_{M\times M}(T)$ supports the image of the diamond at $\gamma_{M\times M}'(T)$. In particular, \[\gamma_{M\times M}'(T)\in \{ s(\gamma_1'(t_1),0) +(1-s)(0,-\gamma_2'(t_2))\mid s\in[0,1]\}\subseteq \partial \rho_{M\times M}(C_{M\times M}), \]
and therefore we deduce that $\lambda_{M\times M}$ annihilates the tangent space of 
\[ \{ s(\gamma_1'(t_1),0) +(1-s)(0,-\gamma_2'(t_2))\mid s\in[0,1]\}. \]
Then, the fact that $\lambda_i$ also supports $\rho(C)$ at $\gamma_i'(t_i)$ translates to
\[H_f(\lambda_1)=\lambda_1(\gamma_1'(t_1))=\lambda_2(\gamma_2'(t_2))=H_f(\lambda_2),\] 
which is the desired property.
\end{proof}

\subsection{Reflection through minimisation}

We close this Section with a couple of examples in which the reflection law follows from the minimising properties of geodesics. In this Subsection we assume the sub-Finsler norm to be reversible, i.e., $f(v)=f(-v)\;\forall v$.

\subsubsection{Balls and Ellipsoids}

The first examples in this direction are spheres and ellipsoids:
\begin{definition}
Fix a reversible sub-Finsler manifold $(M,\xi,f)$ and a positive number $T$. Then, the \textbf{ball} of radius $T$ and center $q \in M$ is:
\[ B(q,T) := \{x\mid d(x,q)\leq T\} \]
The \textbf{ellipsoid} with focal points $q_1,q_2 \in M$ and length $T$ is:
\[ E(q_1,q_2,T) := \{x\mid d(x,q_1)+d(x,q_2)\leq T\} = \bigcup_{t \in [0,T]}  B_t(p)\cap  B_{T-t}(q). \]
\end{definition}
In the Riemannian setting, one would require $T$ to be smaller than the injectivity radius. In the reversible sub-Finsler setting, pathological behaviors can happen even under that assumption. It is known that the boundary of the ball is never smooth \cite[Theorem 1]{Agr} and geodesics of length $T$ starting at $q$ need not reach the boundary of $B(q,T)$ (regardless of how small $T$ is). It is not even known in full generality if a sufficiently small sub-Finsler ball is a topological ball (see~\cite[Section 10.3]{M}).

\begin{lemma}
Let $\gamma: [0,T] \to B(q,T)$ be a minimiser of length $T$ connecting $q$ to a smooth point $\gamma(T) \in \partial B(q,T)$, regular for the boundary distribution. Then:
\begin{itemize}
    \item $\gamma$ is a normal geodesic.
    \item $\gamma$ is invariant under reflection at $p$.
\end{itemize}
\end{lemma}
\begin{proof}
Due to reversibility, both $\gamma$ and its reverse $\widetilde\gamma(t) := \gamma(T-t)$ satisfy PMP. If $\lambda$ is a lift of $\gamma$ satisfying PMP, then $\widetilde\lambda(t) := -\lambda(T-t)$ lifts $\widetilde\gamma$ and also satisfies PMP. Further, the concatenation of $\gamma$ with $\widetilde\gamma$ also satisfies PMP for the control problem of going from $q$ to itself with one reflection. This implies that Equation (\ref{eq:reflectionLaw}) must hold, allowing us to deduce that $\lambda(T)$ reflects to $-\lambda(T)$, which is equivalent to $\lambda(T)$ annihilating the sphere. From this, and using that $\gamma(T)$ is a regular point, we deduce that $\lambda(T)$ cannot annihilate $\xi$, proving that $\gamma$ is normal.
\end{proof}

In symmetric situations it may be the case that $\gamma$ can be continued over the time interval $[T,3T]$, first going back to $q$ and then reaching again the boundary $\partial B(q,T)$. If this second point of contact is also regular, this yields a 2-bounce orbit. In particular:
\begin{corollary}
Let $(M,\xi,g)$ be a Carnot group. Then, any minimising curve between the identity and the smooth, regular locus of the unit sphere extends to a 2-bounce orbit.
\end{corollary}
\begin{proof}
Let $\gamma: [0,1] \to U$ be the geodesic in question. Since the inverse map of the group is an isometry, our curve extends to a minimiser $\gamma: [-1,1] \to U$ with $\gamma(-1) = \gamma(1)^{-1}$, which must also be a point of smoothness, since the inversion identifies the unit sphere with itself. The claimed $2$-bounce orbit consists of $\gamma$ concatenated with its reversal.
\end{proof}

One can similarly prove:
\begin{lemma}
Let $q$ be a point in which $\partial E(q_1,q_2,T)$ is smooth. Let $\gamma_1$ be a geodesic of length $t$ from $q_1$ to $q$ such that $q \in \partial B(q_2,T-t)$. Then $\gamma_1$ reflects to a geodesic $\gamma_2$ connecting $q$ to $q_2$.
\end{lemma}

\begin{question}
If a geodesic in the ball or ellipsoid connects the center or a focal point to a smooth regular point of the boundary and is \emph{not} length minimizing, can one say where it is reflected? \hfill$\triangle$
\end{question}

\subsubsection{Attainable sets and wavefronts}

The key fact we exploited for the ball is that it is the set of points that can be reached from a given point in time at most $T$. We now generalise this. We introduce the notation:
\begin{definition}\label{def:wavefront}
Let $N \subset (M,\xi,f)$ be a smooth submanifold and $\phi_f^T$ the sub-Finsler cogeodesic flow at time $T$. Then:
\begin{itemize}
    \item its \textbf{attainable set} $\Acal_T(N)$ at time $T$ is the set of points that can be reached from $N$ using an admissible curve defined over the time interval $[0,T]$.
    \item its \textbf{(normal) wavefront} $\wavefront_T(N)$ is the set of points $\pi \circ \phi_f^T(\Ann(TN) \cap S_f)$, i.e., the points that can be reached by following the normal geodesics with initial momentum annihilating $TN$ and of unit length.
\end{itemize}
\end{definition}
Note that in our definition of wavefront we are once again embracing the Hamiltonian viewpoint and disregarding strict abnormals. We first remark:
\begin{lemma}
Let $p\in \partial\Acal_T(N)$ be smooth and regular for the restriction of $\xi$. Then, any geodesic $\gamma$ of length $T$ connecting $\gamma(0) \in N$ with $\gamma(T) = p$ reflects to another geodesic of length $T$ connecting $p$ to $N$. 
\end{lemma}
\begin{proof}
The concatenation of $\gamma$ with its reflection is a curve that minimises going from $N$ to $\partial\Acal_T(N)$ and back. Then, the smoothness and regularity assumptions imply that $\gamma$ and its reflection are normal geodesics.
\end{proof}

We can provide a more detailed analysis. According to Proposition \ref{prop:PMPboundary}, minimising curves between two submanifolds always lift to momentum curves satisfying PMP and annihilating their tangent spaces at the endpoints. From this, it follows:
\begin{lemma} \label{lem:wavefrontSubset}
Let $N$ be a smooth submanifold with induced distribution $\xi_N := \xi \cap TN$ (possibly singular). Assume $\dim(N) \geq \operatorname{codim}(\xi)$. Then: 
\[ \partial\Acal_T(N) \quad\subset\quad \wavefront_T(N) \,\cup\, \left( \bigcup_{\textrm{$q$ singular point of $\xi_N$}} \partial\Acal_T(p) \right). \]
\end{lemma}
\begin{proof}
Let $\gamma$ be a minimising curve between $\gamma(0) = q \in N$ and $\gamma(T) = p \in \partial\Acal_T(N)$. Suppose $q$ is not a singular point of $\xi_N$. Then, there is a lift $\lambda$ satisfying PMP and annihilating $T_pN$. From the annihilation property, and our dimension assumptions, we deduce that $\lambda(0)$ cannot annihilate $\xi$, so $\gamma$ is thus normal and $p$ lies in the (normal) wavefront. Alternatively, $q$ is singular and $p$ then lies in the sphere at distance $T$ from it.
\end{proof}

\begin{remark}
In particular, if there are no strict abnormals or if $\xi_N$ has no singular points, we can restrict our attention to the wavefront. \hfill$\triangle$
\end{remark}
\begin{remark}
Note that $\Ann(TN) \cap S_f$  is not a smooth bundle over $N$: the rank of its fibres drops along the singularities of $\xi_N$. It follows that the topology of $\wavefront(N)$ depends on $\xi_N$. \hfill$\triangle$
\end{remark}

For hypersurfaces, $\Ann(TN) \cap S_f$ is a double cover of $N^\circ$ (i.e., $N$ minus its singularities). From the Lemma it follows: 
\begin{corollary} \label{cor:generalWaveFront}
Let $N$ be a codimension-$1$ submanifold. Let $q\in \partial\Acal_T(N)$ be a point of smoothness contained in the wavefront. Let $\gamma$ be a minimiser connecting $\gamma(0) \in N$ with $\gamma(T) = q$. Then the following conditions are equivalent:
\begin{itemize}
    \item $\gamma$ is normal.
    \item $\gamma(0)$ is a regular point of the distribution $TN \cap \xi$.
    \item $q$ is a regular point of the induced distribution $T\Acal_T(N) \cap \xi$.
\end{itemize}
And they all imply that $\gamma$ reflects to itself.
\end{corollary}
\begin{proof}
Let $\lambda$ be a conormal lift satisfying PMP. $\lambda(0)$ and $\lambda(T)$ must annihilate the corresponding tangent spaces at the endpoints, so all three conditions are equivalent to the fact that, at $0$ and $T$, the annihilators of $N$ and $\partial\Acal_T(N)$ are not contained in the annihilator of $\xi$.

For the last claim, Lemma \ref{lem:wavefrontSubset} implies that $\lambda(0)$ is the unique momentum in the unit cylinder annihilating $T_{\gamma(0)}N$. Invoking the reversibility of the Finsler norm, $\lambda(0)$ is readily seen to reflect to $-\lambda(0)$. Reasoning in the same manner at $T$ concludes the proof.
\end{proof}

\section{Billiard tables in the Heisenberg group} \label{sec:specialCase}

In this last Section we focus on the standard contact structure on $\RR^3$ endowed with the pullback of the euclidean metric on the plane. Its sub-Riemannian geodesic flow is well-understood: it corresponds to the classic Dido problem that asks to maximise the area bounded by a curve of given length.

For some concrete examples of simple billiard tables we analyse their reflection laws, which allows us to provide a (partial) description of their periodic orbits. 

\subsection{The standard contact structure}

\begin{definition}
The (radially symmetric) \textbf{standard contact structure} $\xi_\std := \ker(\alpha_\std)$ in $\RR^3$ is defined as the kernel of the standard contact form
\[ \alpha_\std := dz - \frac{1}{2}(ydx-xdy)= dz - \frac{1}{2}r^2d\phi. \]
\end{definition}
Up to diffeomorphism, this is the unique tight contact structure on $\RR^3$. 

We may then define a metric $g_\std$ on $\xi_\std$ as the pullback of the euclidean metric on the $(x,y)$-plane by the projection along the $z$-direction:
\begin{definition}
The \textbf{standard metric} on $\xi_\std$ is the restriction of the bilinear form:
\[ g_\std := dx \otimes dx + dy \otimes dy \]
to the contact structure.
\end{definition}

One then observes that $g_\std$ singles out $\alpha_\std$ as a preferred contact \emph{form} for $\xi_\std$. Indeed, it is the unique $\alpha$ such that $d\alpha|_{\xi_\std}$ is the area form on $\xi_\std$ defined by $g_\std$. Then:
\begin{lemma} \label{lem:ReebField}
The Reeb vector field of $\alpha_\std$ is $R_\std = \partial_z$. It is Killing for the resulting sub-Riemannian manifold $(\RR^3, \xi_\std, g_\std)$.
\end{lemma}

The coframe $\{\alpha,dx,dy\}$ induces coordinates $a\alpha + bdx + cdy$ on the cotangent bundle $T^*\RR^3$. In these coordinates, the square of maximised Hamiltonian reads 
\[ H_\std(x,y,z,a,b,c) = b^2 + c^2. \]
Noether's theorem tells us that symmetries correspond to integrals of motion of the sub-Riemannian geodesic flow. The Reeb symmetry translates into the statement:
\begin{lemma}
The Hamiltonian flow of $H_\std$ is tangent to the level sets of the dual coordinate $a$.
\end{lemma}


\subsubsection{Symmetry group}

We now describe the isometry group $\sym := \sym(\RR^3,\xi_\std,g_\std)$; its action on the space of tables will be relevant later on. First, note that for every $v=(v_1,v_2,v_3)\in\RR^3$ there is a symmetry
\begin{align}\label{shearmap}
\varphi_v(x,y,z)&\mapsto (x+v_1, y+v_2, z+v_3+v_2x-v_1y),
\end{align}
mapping $0$ to $v$, which means that $\sym$ transitive. The subgroup consisting of the maps $\varphi_v$ is isomorphic to the three dimensional Heisenberg group $(\{\varphi_v\},\circ)=\Heis_3$. The group is even larger, since the rotations about the $z$-axis are also symmetries. It follows that:
\begin{lemma}
$\sym(\RR^3,\xi_\std,g_\std) = \Heis_3 \ltimes \SO(2)$.
\end{lemma}
\begin{proof}
It is sufficient to show that the symmetry group cannot be larger than the claimed semidirect product. First note that any symmetry of $(\RR^3,\xi_\std,g_\std)$ must preserve the contact form $\alpha_\std$ and thus the Reeb field $R_\std$, according to the discussion preceding Lemma \ref{lem:ReebField}.

Given a linear isometry 
\[ A: ((\xi_\std)_0,g_\std) \cong ((\xi_\std)_v,g_\std), \]
we observe that there is exactly one element $\varphi \in \Heis_3 \ltimes \SO(2)$ realising it. We claim that any other symmetry inducing $A$ must agree with $\varphi$. Indeed, if two symmetries restrict to $A$, we deduce, recalling the fact that $\alpha_\std$ is preserved, that they yield the same isomorphism $T^*_0\RR^3 \cong T^*_v\RR^3$. Since symmetries map geodesics to geodesics, this identification of cotangent fibres extends uniquely to a global symmetry.
\end{proof}

\subsubsection{The geodesic flow}\label{descriptionofgeodesics}

In order to study the horizontal curves of $\xi_\std$, we introduce:
\begin{definition}
The \textbf{Lagrangian projection} is the map:
\begin{align*}
\pi:\RR^3 &\quad\to\quad \RR^2 \\
(x,y,z) &\quad\mapsto\quad (x,y).
\end{align*}
\end{definition}
Since $\alpha_\std$ is non-vanishing on the fibers of $\pi$, it defines a connection on $\RR^3$, seen as as an $\RR$-bundle over the plane. Curves in the base can be lifted uniquely (up to the choice of a starting point) to horizontal curves. Indeed, the condition $\gamma^*\alpha_\std=0$ translates to
\[ dz(\dot\gamma)=\frac{1}{2}r^2d\phi \circ d\pi(\dot\gamma). \]

Combining this with the fact that $d[\frac{1}{2}r^2d\phi] = dx\wedge dy$ is the standard area form in $\RR^2$, we see by Stokes' theorem that the variation in the $z-$coordinate of $\gamma$ coincides with the signed area $\Gamma$ enclosed by $\pi\circ\gamma$. If $\pi\circ\gamma$ is not closed, then the area enclosed is understood by radially connecting the endpoints to the origin. These additional segments do not contribute to the vertical displacement. We conclude:
\begin{equation}\label{zvariation}
    \Delta z = \int_{\pi \circ\gamma} \frac12r^2d\phi = \operatorname{Area}(\Gamma)
\end{equation}

The (local) length minimizing property may thus be restated as the problem: Find a path $\pi\circ\gamma$ in $\RR^2$ connecting two prescribed points, surrounding a prescribed signed area with minimizing length. This is equivalent to the classical isoperimetric problem of Dido. Solutions are well-known:
\begin{lemma}
A horizontal curve $\gamma$ is a sub-Riemannian geodesic of $(\RR^3,\xi_\std,g_\std)$ if and only if its projection $\pi \circ \gamma$ is a circular segment in the plane.
\end{lemma}

For completeness (and in order to set up notation), we will provide a proof using the Hamiltonian viewpoint now.

\begin{remark} \label{rem:convergingToConstantExample}
We stated in Remark \ref{rem:convergingToConstant} that a sequence of geodesics, defined over the whole of $\RR$, can converge to a constant curve. Using the Lemma is easy to construct such a sequence: take geodesics $\gamma_i$ with projections $\pi \circ \gamma_i$ describing circles of radius $1/i$ centered at the origin. As $i$ goes to zero, $\gamma_i|_{[0,T]}$ has length $T$ but its projection bounds area $O(1/i)$. It follows that their $C^0_{loc}$-limit is the origin. \hfill$\triangle$ 
\end{remark}

\subsubsection{Hamiltonian description} \label{sssec:HamiltonianR3}

The tautological Liouville 1-form $\lambda$ on $T^*M$ may be written, in terms of our coframe, as $\lambda = a\alpha_\std+bdx+cdy$. We write
\[ X :=\partial_x-\frac{1}{2}y\partial_z,\quad Y := \partial_y + \frac{1}{2}x\partial_z, \quad R_\std = \partial_z, \]
for the dual frame. With some abuse of notation, we consider the framing $\{R_\std,X,Y,\partial_a,\partial_b,\partial_c\}$ of $T^*M$ and we compute:
\begin{align}
d\lambda & = da\wedge\alpha_\std + db\wedge dx + dc\wedge dy + ad\alpha,\\
H & = \frac12 (b^2+c^2),\qquad dH = bdb + cdc.
\end{align}

\begin{align*}
\iota_{\partial_a}d\lambda \quad=\quad & \alpha_\std,\\ \iota_{\partial_b}d\lambda \quad=\quad & dx,\\
\iota_{\partial_c}d\lambda \quad=\quad & dy,\\
\iota_Xd\lambda      \quad=\quad & -db + a d\alpha_\std(X,\cdot) = -db
+ a d\alpha_\std(X,Y) dy + a\underbrace{d\alpha_\std(X,R_\std)}_{=0}\alpha, \\
\iota_Yd\lambda      \quad=\quad & -dc + a d\alpha_\std(Y,\cdot) = -dc
+ a d\alpha_\std(Y,X) dx + a\underbrace{d\alpha_\std(Y,R_\std)}_{=0}\alpha, \\
\iota_{R_\std}d\lambda \quad=\quad & -da+a\underbrace{d\alpha_\std(R_\std,\cdot)}_{=0}.
\end{align*}


The preceding computations then show:
\begin{align}
X_H (x,y,z,b,c,a) = bX + cY
- a \underbrace{d\alpha(X,Y)}_{=1}(b\partial_c - c\partial_b)
\end{align}
That is, if $\gamma$ is a Hamiltonian orbit of $X_H$, the vector $(b,c) = b\partial_x + c\partial_y$ is the velocity of its projection $\pi \circ \gamma$. Similarly, $a(c,-b)$ is the acceleration. Therefore, projections of solutions are circles of radius $\frac{b^2+c^2}{a}$. If $a=0$, these are straight lines.


\subsection{Tables with vertical walls} \label{ssec:verticalWalls}

We now consider the case in which the billiard table is an infinite cylinder $Z = U\times \RR$, where $U$ is a billiard table in $\RR^2_{x,y}$.

The first observation is that the Reeb vector field foliates the boundary $\partial Z$ and, as such, the intersection $T\partial Z \cap \xi_\std$ is a non-degenerate line field. It also follows that lines parallel to $\Ann T\partial Z$ are tangent to the level sets of the integral of motion $a: T^*\RR^3 \to \RR$. We deduce:
\begin{proposition}\label{prop:lift}
Billiard trajectories $\gamma$ in $Z$ are lifts of magnetic billiard trajectories $\pi \circ \gamma$ in $U$. The momentum $a$ of $\gamma$ coincides with the magnetic curvature of $\pi \circ \gamma$.
\end{proposition}
\begin{proof}
Since $\Ann T\partial Z$ is parallel to $a$, the latter is conserved at reflection points. The claim then follows by the discussion in the previous Subsection \ref{sssec:HamiltonianR3}.
\end{proof}

\subsubsection{Reflections}

Suppose a curve $\gamma$ hits the boundary $\partial Z$ at a point $q$ with momentum $\lambda_{\rm in}$. According to the reflection law Lemma~\ref{lem:sympReflectionLaw} and surrounding discussion, there are two possibilities.

The non-degenerate case takes place when the line $L$ parallel to $\Ann T_q\partial Z$ passing through $\lambda_{\rm in}$ intersects the cylinder of energy one transversely. Then, the outgoing momentum $\lambda_{\rm out}$ is the unique other intersection.

In the degenerate case $L$ is tangent to the cylinder at $\lambda_{\rm in}$. Observe first that the magnetic curvature $a \circ \gamma$ must be smaller or equal than the curvature of $\partial U$ at $\pi(q)$. If it is strictly smaller, the trajectory can be continued as a geodesic in $Z$ with a tangency with $\partial Z$ (i.e., no reflection happens). When there is equality, the trajectory may be continued as a \emph{gliding trajectory}, as we explain below.

\begin{remark}
In magnetic billiards with given magnetic strength $a$, one says that a table is \emph{convex} if the curvature of the table is greater than $a$. In general, one would define convexity by requiring that no billiard trajectory exists that is tangent to the boundary of the table. This is never the case in the sub-Riemannian setting, precisely because one can always consider orbits of arbitrarily large momentum tangent to the boundary. \hfill$\triangle$
\end{remark}

\begin{remark}
One could consider the case when $U$ has corners and, in particular, the setting in which $U$ is a polygon. There are various possible decisions: A billiard curve hitting a corner might stop, or it might continue as long the difference of ingoing and outgoing momentum lies in the normal cone of the corner. In this article, we do not elaborate on this and we refer to possible future works. \hfill$\triangle$
\end{remark}

\subsubsection{Periodic orbits}\label{verticalperiodic}

Proposition~\ref{prop:lift} shows that every periodic sub-Riemannian billiard orbit in $Z$ is the lift of a periodic magnetic billiard orbit in $U$. The converse is not true, since the lift might not close up. To know if the lift of a periodic magnetic billiard orbit closes, it suffices to consider one period and check whether the variation $\Delta z$ vanishes, since subsequent periods lift to the same variation in $z$. Using Equation~\ref{zvariation} we are led to the criterion:
\begin{lemma}
Periodic billiard orbits in $Z$ are exactly the lifts of periodic magnetic billiard orbits in $U$ enclosing zero area.
\end{lemma}

\subsection{The standard cylinder} \label{ssec:standardCylinder}

We now focus on the setting where $U = \DD^2$ is the unit disc and $Z$ is the standard cylinder. Note that there is an $\RR\times \SS^1$-symmetry of the table $Z$ given by $z$-translation and rotation about the $z$-axis. These symmetries act transitively on the boundary $\partial Z$. We use this to characterise its periodic orbits.

\subsubsection{Periodic orbits}

Let $\gamma$ be an $n$-periodic sub-Riemanniann billiard orbit and $\pi\circ\gamma$ its projection to the plane. Reflections are uniquely determined by the angle of incidence and the magnetic momentum $a$. Since $a$ is an invariant, all the reflection points of $\gamma$ have the same angle of incidence. In particular, the circular arcs between the reflection points of $\pi\circ\gamma$ are all isometric by a rotation. As such:
\begin{lemma}
The reflection points of $\pi\circ\gamma$ lie in a regular $n$-gon. Each of its circular arcs must enclose zero area.
\end{lemma}
Do note that two consecutive reflection points of $\pi \circ \gamma$ do not need to be consecutive in the $n$-gon.

Let us assume that such an arc connects two boundary points that are separated by an angle of $2\varphi$, as seen from the origin. Let $2\psi$ be the angle between the two points, as seen from the center of curvature of their arc. The radius of curvature is then $\rho=\frac{\sin\varphi}{\sin\psi}$. We depict this in Figure \ref{fig:ngon}. 

\begin{figure}[h]
\centering
\includegraphics[width=0.5\textwidth]{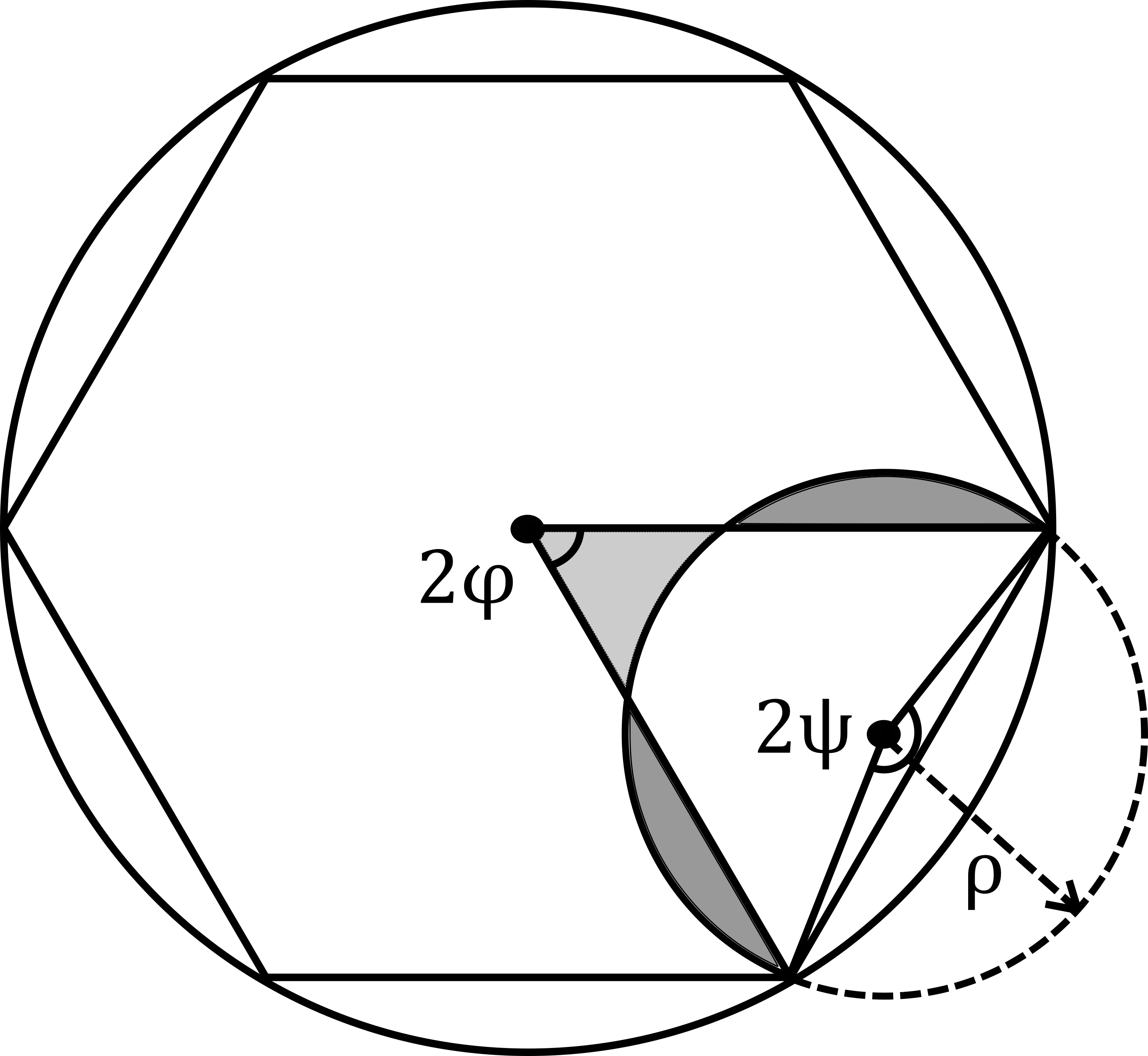}
\caption{The $n$-gon from Lemma \ref{lem:ngon}, with $n=6$, and an arc connecting two of the boundary points (in this case, adjacent also in the hexagon, so $m=1$). The lightly shaded region must have the same area as the sum of the dark regions.}\label{fig:ngon}
\end{figure}

The condition that this arc bounds zero area reads:
\begin{align}\label{deltaz}
0 = \Delta z = \sin\varphi\cos\varphi - \left(\frac{2\pi-2\psi}{2\pi}\pi\rho^2 + \sin\psi\cos\psi\rho^2\right),
\end{align}
which, after dividing by $\sin^2\varphi\neq0$, is equivalent to
$$\operatorname{cotan}\varphi=\operatorname{cotan}{\psi} + (\pi-\psi)\operatorname{cosec^2}(\psi).$$

The right-hand side is a function of $\psi$ that is a decreasing homeomorphism from $(0,\pi)$ to $\RR_{> 0}$. As such, for each value $\varphi\in(0,\frac\pi2)$, there is a unique solution. We deduce:
\begin{lemma} \label{lem:ngon}
Given any $n$-gon inscribed in $\DD^2$, and any coprime integer $m$, there exists, up to $z$-translation, a unique billiard trajectory $\gamma$ such that:
\begin{itemize}
    \item $\pi \circ \gamma$ is comprised of $n$ circular segments with vertices in the $n$-gon,
    \item consecutive vertices of $\pi \circ \gamma$ are obtained by a rotation of angle $\frac{2\pi m}n$.
\end{itemize}
\end{lemma}

We then observe:
\begin{proposition} 
There are no periodic sub-Riemannian billiard orbits in the interior of $Z$. In particular, all the periodic billiard trajectories of $Z$ are described by Lemma \ref{lem:ngon}.
\end{proposition}
\begin{proof}
Any billiard trajectory $\gamma$ avoiding $\partial Z$ must be one of the tilted helices described at the end of~ Subsection \ref{descriptionofgeodesics}. It projects to a closed circle $\pi \circ \gamma$ disjoint from the boundary of $U$. Since $\pi \circ \gamma$ bounds non-zero area, the claim follows.
\end{proof}

We remark that the standard cylinder $Z$ is \textbf{integrable}: The sub-Riemannian billiard orbits that hit the boundary are, up to isometry, completely described by the angle of reflection and the magnetic momentum. Similarly, all trajectories missing the boundary are uniquely determined by the magnetic momentum and the distance of the axis of the helix to the $z$-axis.

\subsubsection{Caustics}

We recall:
\begin{definition}
Given a sub-Riemannian billiard table $Z \subset (M,\xi,g)$, a \textbf{caustic} is a submanifold $C \subset Z$ such that any billiard trajectory tangent to $C$ at one point, is in fact tangent to $C$ at exactly one point in-between every two consecutive reflection points.
\end{definition}

As a result of its abundant symmetries, the standard cylinder is fibered by caustics:
\begin{proposition}
Every smaller cylinder $C$ centered around the $z$-axis is a caustic for the standard cylinder $Z$.
\end{proposition}
\begin{proof}
It follows from the invariance of $C$ under the symmetry group $\RR\times \SS^1$.
\end{proof}

\begin{question}
Are there other examples of sub-Riemannian billiard tables with caustics? Conversely: does the existence of a caustic imply that the table is conformal to $Z$? \hfill$\triangle$
\end{question}

\subsubsection{Gliding and creeping trajectories}\label{subsubsec:creeping}

In Subsection \ref{sssec:glide} we showed that gliding trajectories are reparametrisations of boundary geodesics. A question that we left open is whether every such boundary geodesic is actually a glide orbit. It is easy to construct examples of classical billiards where this is not the case. We explore this in our setting, proving:
 \begin{lemma}
The $C^1$-gliding trajectories on the standard cylinder $Z$ are exactly the helices of slope $1$ tangent to $\partial Z$. I.e., the leaves of the boundary foliation $\xi_\std \cap T\partial Z$.
\end{lemma}
 \begin{proof}
The leaves of the characteristic foliation, suitably parametrised, are already ambient geodesics. They are approximated as Lipschitz curves by the constant sequence of geodesics. 

Conversely, we know from Proposition~\ref{prop:boundaryGeodesic} that $C^1$-gliding trajectories are boundary geodesics. Since the only geodesics of $(\partial Z, \xi_\partial, f)$ follow the boundary foliation, the claim follows.
 \end{proof}
 We remark that the same statement holds true for any billiard table whose boundary is foliated by ambient geodesics. Similarly:
 \begin{proposition}
 The creeping trajectories on $Z$ are exactly the helices tangent to $\partial Z$, including the horizontal circle and the vertical line.
 \end{proposition}
 \begin{proof}
We first approximate the vertical lines; the reader should compare to Remark \ref{rem:convergingToConstantExample}. We consider geodesics $\gamma_\rho$ tangent to the boundary and of magnetic radius of curvature $\rho$ smaller than the radius of the cylinder $Z$; their reflections are thus all trivial and happen with a period of $2\pi\rho$. As described in Subsection \ref{shearmap}, every period they gain $z$-coordinate $\Delta z$ equal to the area $\pi\rho^2$ of the projected circle $\pi\circ\gamma_\rho$. If one lets $\rho$ go to zero, the curves will converge, when appropriately reparametrized, to a vertical line:
$$\gamma_\rho\left(\frac 1\rho t\right)\stackrel{C^0}{\longrightarrow} (x_0,y_0,z_0+\frac12t).$$ 
This convergence is not $C^1$ and the limit is not an admissible curve of the boundary distribution.

In a similar manner, we now approximate helices on the boundary of fixed slope $\sigma\neq 1$. We fix some point $(r_0,\varphi_0,z_0) \in \partial Z$, expressed in cylindrical coordinates. Recall Equation~(\ref{deltaz}) and the angles $\varphi$ and $\psi$ introduced there. We define $\gamma_{\varphi, \psi}$ as the billiard trajectory through $(r_0,\varphi_0,z_0)$ whose projection $\pi \circ \gamma_{\varphi, \psi}$ consists of circular arcs, where the angle between consecutive intersections with the boundary, measured from the centre of $U$, is $\varphi$, and measured from the centre of the arc is $\psi$. Its radius of curvature is $\rho=\frac{\sin\varphi}{\sin\psi}$ and thus the slope between reflection points is equal to
\[ \sigma=\frac{\Delta z}{\varphi}= \frac1\varphi\left( \sin\varphi\cos\varphi - \left(\frac{2\pi-2\psi}{2\pi}\pi\rho^2 + \sin\psi\cos\psi\rho^2\right)\right). \]
We fix $\varphi$ and $\sigma$ and we solve for $\psi$ in order for the equation to hold. Rearranging yields:
\[ {\cotan}(\varphi)-\frac\varphi{\sin^2\varphi}\sigma=\frac1{\sin^2\psi}\left(\pi-\psi+\sin\psi\cos\psi\right). \]
Since the left-hand side is fixed and the right-hand side is a monotone decreasing homeomorphism $(0,2\pi)\to\RR$, there is a unique solution $\psi(\varphi)$. We now take the limit $\varphi \to 0$. The left-hand side behaves like $(1-\sigma)\frac1\varphi$. The right-hand side must also be unbounded, which implies that $\sin\psi\to 0$ and thus $\psi\to 0$. The asymptotic behavior of the equality then reads:
\[ (1-\sigma)\frac1\varphi\sim \pi\frac 1{\psi^2}. \]
Therefore, up to a constant factor, we have that $\psi\sim \sqrt\varphi$ and $\rho \sim\sqrt\varphi$ as $\varphi\to 0$. In particular, the circular segments in-between reflections have length $2\rho(\pi-\psi) \sim 2\pi\sqrt\varphi$. This proves that there is an appropriate reparametrization parameter $c_\varphi \sim \frac{1}{\sqrt{\varphi}}$ such that:
\[ \gamma_{\varphi,\psi(\varphi)}(c_\varphi t) \to_{C^0} (r_0,\varphi_0+t,z_0+\sigma t). \]

We have thus proven the existence of all claimed creeping trajectories. For the converse we claim that there are no others. Recall that the symmetry group of the cylinder is $\sym(Z)=\RR\times S^1$, which is transitive. Billiard trajectories are completely determined, up to this symmetry, by the angle of reflection and the transverse momentum, both of which are preserved under reflections. Given a sequence of (reparametrised) billiard trajectories $\gamma_i$ converging to a creeping orbit $\gamma_\infty$, there is a sequence of subgroups $\Gamma_i \subset \sym(Z)$ preserving $\gamma_i$ set-theoretically; the elements of $\Gamma_i$ are in correspondence with the reflection points of $\gamma_i$. As we take the limit, the $\Gamma_i$ will converge to a $1$-parameter subgroup $\Gamma_\infty$ preserving $\gamma_\infty$. The 1-parameter subgroups of  $\sym(Z)$ generate precisely the curves described above.
\end{proof}

\begin{remark} \label{rem:diffOperators}
It is well-known that the length spectrum of a metric relates to the spectrum of the corresponding Laplacian. This has been studied in the sub-Riemannian setting as well, but our understanding is far from complete. An interesting phenomenon is that, in the contact case, periods of Reeb orbits (and not just geodesic periods) appear in spectrum calculations of certain Dirac operators associated to the sub-Riemannian structure \cite{Sav}. This can be understood as an incarnation of the fact that sub-Riemannian geodesics tend to concentrate along Reeb orbits (up to reparametrisation, just like in our definition of creeping) as we make their momenta diverge. We have seen this in our $\RR^3$ example and it has been announced in full generality in \cite[p. 10]{CVHT}.

Once we consider manifolds with boundary, one expects billiard trajectories and boundary geodesics to relate as well to the spectra of associated differential operators. It is then an intriguing open question to understand the analytical role played by creeping orbits (which may play a role analogous to that of Reeb orbits). \hfill$\triangle$
\end{remark}

\subsection{Vertical planes}

We now consider a vertical half-space as table, with boundary a vertical plane. The same analysis as in the cylinder setting shows that:
\begin{lemma}
Let $A$ be a half-space table in $(\RR^3,\xi_\std,g_\std)$ with boundary a vertical plane at distance $v$ from the origin. Then: \begin{itemize}
    \item Gliding orbits correspond to straight lines of slope $v$.
    \item Creeping trajectories correspond to straight lines.
\end{itemize}
\end{lemma}
\begin{proof}
First note that $v$ is the only invariant of vertical planes up to symmetry; we may assume that $\partial A = \{y=v\}$. Gliding orbits correspond to integral lines of the boundary foliation. If $v=0$, these are horizontal lines. Otherwise we use the symmetry $\varphi_v(x,y,z) = (x,y+v,z+vx)$ mapping the plane $\{y=0\}$ to $\{y=v\}$ (recall Equation \ref{shearmap}). This shows that glide orbits are the lines of slope $v$.

Now we deal with creeping orbits. Given any straight line $L$ in $\{y=0\}$, we produce a sequence of billiard trajectories $\gamma_i$ such that $\pi \circ \gamma_i$ consists of many copies of the same circular arc shifted in the $y$-direction with separation $1/i$. The radius of the arcs is determined uniquely from the slope of $L$, as in the previous Subsection. A subtlety now is that the convergence to creeping orbits is not uniform in the $z$--coordinate: even though all circular arcs look the same in the projection, their lifts have a variation in $z$, at the half-point, that is proportional to the distance to the origin.

To show that all creeping orbits must be straight lines, we use again that the symmetries of the half-space act transitively on its boundary. Lastly, we transfer the statement to $A$ using the map $\varphi_v$.
\end{proof}

\begin{question}
For billiard tables whose symmetry group does not act transitively on the boundary, what do the creeping orbits look like? \hfill$\triangle$
\end{question}

\subsection{A horizontal half-space}\label{sec:horizontalPlane}

 In the previous Subsection we studied vertical planes. Observe then that the shear isometry given in Equation \ref{shearmap} sends horizontal planes to planes with slope $|v|$. It follows that, up to isometry, there are only two distinct types of half-space tables: vertical ones, and horizontal ones. We therefore focus on a table of the form $\RR^2\times \RR_{\geq0}$ with $H := \partial(\RR^2\times \RR_{\geq0})=\RR^2\times \{0\}$. We leave it to the reader to adapt the discussion to the table having $H$ as upper boundary instead.

\subsubsection{Self-reflecting curves}

We use cylindrical coordinates $(r,\theta,z)$. The boundary foliation $\xi_\partial$ along $H$ has a unique singular point in $(0,0,0)$. Everywhere else we have that $\xi_\partial = \langle\partial_r\rangle$, with $TH$ and $\xi_\std$ forming an angle of $\frac r2$. We deduce: 
\begin{lemma}
Let $\gamma$ be a germ of billiard trajectory with $\gamma(0) = (r,\theta,0) \in H$. Let $a \in \RR$ be the ingoing magnetic momentum at the reflection point. Let $\phi \in (0,\pi)$ be the ingoing angle of incidence of $\pi\circ \gamma$ with respect to the radial line. Then:
\begin{itemize}
    \item The outgoing angle of incidence of $\pi\circ \gamma$ with respect to the radial line is $\pi-\phi$.
    \item The outgoing magnetic momentum is $a+4\frac{\sin\phi}r$.
\end{itemize}
\end{lemma}
\begin{proof}
 The cylinder field $S_f$ is parallel to $\alpha_\std := dz-\frac{1}{2}r^2d\theta$. The annihilator of $H$ is spanned by $dz$. We study how the reflection segments parallel to $\Ann(TH)$ intersect $S_f$. In doing so we observe that the angle of incidence changes in the usual manner by reflection with the radial line. The change in magnetic momentum is given by the displacement in $dz$.

The angle between $S_f$ and $\Ann(TH)$ is $\frac{2}{r}$. The longest segment parallel to $\Ann(TH)$ contained in $D_f$ has then length $\frac{4}{r}$. The segment passing through the momentum corresponding to angle of incidence $\phi$ has length $4\frac{\sin\phi}r$.
\end{proof}
That is, magnetic curvature \emph{increases} every time there is a reflection. The opposite is true at walls from above.

\begin{corollary} \label{cor:selfReflectionHorizontalPlane}
Let $\gamma$ be a germ of billiard trajectory at a reflection point with $H$. Then, the following claims are equivalent:
\begin{itemize}
    \item $\gamma$ is invariant under the reflection.
    \item $\gamma$ hits $H$ at radius $r$, with ingoing angle $\phi=\frac\pi2$, and ingoing magnetic momentum $a=-\frac2r$.
\end{itemize}
\end{corollary}
I.e., the only curves that reflect to themselves are those whose projections have radius $1/a = r/2$ and center at distance $r/2$ from the origin.

\subsubsection{Attainable sets}
We can use the previous computations to study the attainable set of $H$ at time $T$. We do this by carefully analysing the wavefront. As one expects, the latter accumulates to $H$ at the origin, the singular point of $\xi_\partial$.
\begin{proposition}
The two components of $\partial \Acal^T(H)$ are graphical over $H$ and smooth away from the points lying vertically over the origin. Additionally:
\[ \partial \Acal^T(H) \quad=\quad \pi \circ \phi_T(\Ann(TH)|_{\{r \,\geq\, 2T/\pi\}}\cap S_f), \]
where $\phi_T$ is the time-$T$ flow of the maximised Hamiltonian.
\end{proposition}
\begin{proof}
By symmetry, we are only concerned with the component with positive $z$-coordinate. According to Lemma \ref{lem:wavefrontSubset}, $\partial \Acal^T(H)$ is a subset of the wavefront of $H$ at time $T$, since contact structures have no singular horizontal curves. We thus restrict our attention to geodesics $[0,T] \to \RR^3$ with initial momentum in $\Ann(TH) \cap S_f$. We abuse notation and write 
\[ \phi_t: H \setminus 0 \quad\to\quad \RR^3 \]
for the time $t$-flow of the unique geodesic of this form starting at the given point in $H$ minus the singular point.

Given $(x_0,y_0,0)\in H$, we want to compute the endpoint $\phi_T(x_0,y_0,0) = (x_1,y_1,z_1)\in\partial\Acal^T(H)$. Corollary~\ref{cor:generalWaveFront} states that the trajectory $\gamma(t) = \phi_t(x_0,y_0,0)$ must reflect to itself at $t=0$, so Corollary \ref{cor:selfReflectionHorizontalPlane} implies that $\pi\circ\gamma$ is a circular segment with diameter $r:=\|(x_0,y_0)\|$ and center $M:=\frac{1}{2}(x_0,y_0)$, containing $(x_1,y_1)$. We write $\psi$ for the total angle described by $\pi\circ\gamma([0,T])$ with respect to $M$; modulo $2\pi$, it agrees with $\measuredangle((x_0,y_0),M,(x_1,y_1)))$.

We first observe that $\psi r/2=T$. The $z$-coordinate is then determined by the signed area bounded by the circular arc:
\[ z_1(r) = \operatorname{Area}(\pi\circ\gamma|_{[0,T]})  = \dfrac{r^2}{8}(\psi + \sin(\psi)) = \dfrac{rT}{4} + \dfrac{r^2}{8}\sin(2T/r). \]
Using that $\measuredangle((x_0,y_0),(0,0),(x_1,y_1))=\psi/2$, we can compute the projected distance to the origin at the upper reflection point:
\[ R(r) :=\|(x_1,y_1)\| = r \|\cos(\psi/2)\| = r \|\cos(T/r)\|.   \]
That is, due to the radial symmetry, we can treat both $R$ and $z_1$ as functions in $r$.

This shows that if $T/r\in \pi\ZZ +\pi/2$, then $R=0$. In other words, the circle in $H$ of radius $r_k :=\frac{2T}{(2k+1)\pi}$, $k \in \ZZ^{\geq0}$, is mapped by $\phi_T$ to the point $(0,0,\frac{T^2}{2(2k+1)\pi})$, which is thus part of the cut locus of the distance function to $H$. We claim that $\partial \Acal^T(H)$ is the image of the set $H_0 := \{r_0 \leq r\} \subset H$ under $\phi_T$. The consecutive bands $H_k := \{r_k \leq r \leq  r_{k-1}\}$, $k>0$, will lie below $H_0$. In order to prove this, we revert the problem: we list the points $(x_0,y_0)$ that are mapped by $\phi_T$ to the vertical line over some given $(x_1,y_1)$ and see which one is highest. We write $R_1 = ||(x_1,y_1)||$.

First, note that $R(r)\leq r$ and thus $R^{-1}(R_1)$ is always a finite set. If $R_1 >r_0=2T/\pi$, then there is only one preimage and there is nothing to prove. For $R_1 \in (r_{k_0},r_{k_0-1})$, there is no preimage in $H_k$ for $k>k_0$, there are 0, 1 or 2 preimages in $H_{k_0}$, there are 2 preimages in $H_k$ for $0<k<k_0$ and there is one preimage in $H_0$.

Using 
\[ r^2\sin(2T/r) = 2r^2\sin(T/r)\cos(T/r) = \sgn(\sin \psi) 2\sqrt{(r^2-R^2_1) R_1^2}, \]
we find a new expression for $z_1$:
\[z_1 = \frac{rT}4  + \sgn(\sin\psi)  \frac 14 \sqrt{(r^2-R_1^2) R_1^2},\]
which we must maximise among the preimages of $R_1$. Note that $z_1\leq \frac {rT}4  + \frac 14 \sqrt{(r^2-R_1^2) R_1^2}$, which is strictly monotone and thus maximised by the largest $r$. But this is the preimage in $H_0$, which corresponds to $\psi\in(0,\pi)$, where $\sgn(\sin\psi)=+$. Therefore, the preimage in $H_0$ also maximises $z_1$, which proves the claim.  

The number $r$ found is then the unique preimage in $R^{-1}(R_1)$ that is larger or equal than $2T/\pi$. By the inverse function theorem we can then define $r = R^{-1}$, away from $R_1=0$, as a smooth function of $R_1$. We expand it at the origin:
\[ r(R_1)= \frac {2T}\pi + \frac\pi{2T} R_1+O(R_1^2), \]
which in turn means that 
\[ z_1 =  \frac {T^2}{2\pi} + \left(\frac \pi8 + \frac T{2\pi}\right) R_1 + o(R_1), \]
and thus $R_1=0$ is a non-smooth point.
\end{proof}

\begin{remark}
After the first version of this paper was uploaded to arXiv, the authors found out that a similar computation had already appeared in the literature \cite[Theorem 1.8]{RR} recently. The result, by Rizzi and Rossi, is dual to ours: it provides instead an explicit formula for the distance function to $H$.\hfill$\triangle$
\end{remark}

Using the Lemma and Corollary \ref{cor:generalWaveFront} we deduce:
\begin{corollary}
Consider the table with boundaries $H$ and $\partial \Acal^T(H)$. Then:
\begin{itemize}
    \item $\partial \Acal^T(H)$ is, away from its singularity over the origin, transverse to the contact structure.
    \item Each geodesic starting at $\partial \Acal^T(H)$ minus the singularity, with unit conormal momentum, extends to a $2$-bounce orbit whose other endpoint is in $\{r \,>\, 2T/\pi\} \subset H$.
\end{itemize}
\end{corollary}

\subsubsection{Geodesics approaching the critical point}\label{sssec:critical}

At the critical point, the reflection law is not well-defined. However, we can study the behavior of a sequence of billiard trajectories as their bouncing points approach it. This is depicted in Figure \ref{fig:critical}.

\begin{figure}[h]
\centering
\includegraphics[width=0.9\textwidth]{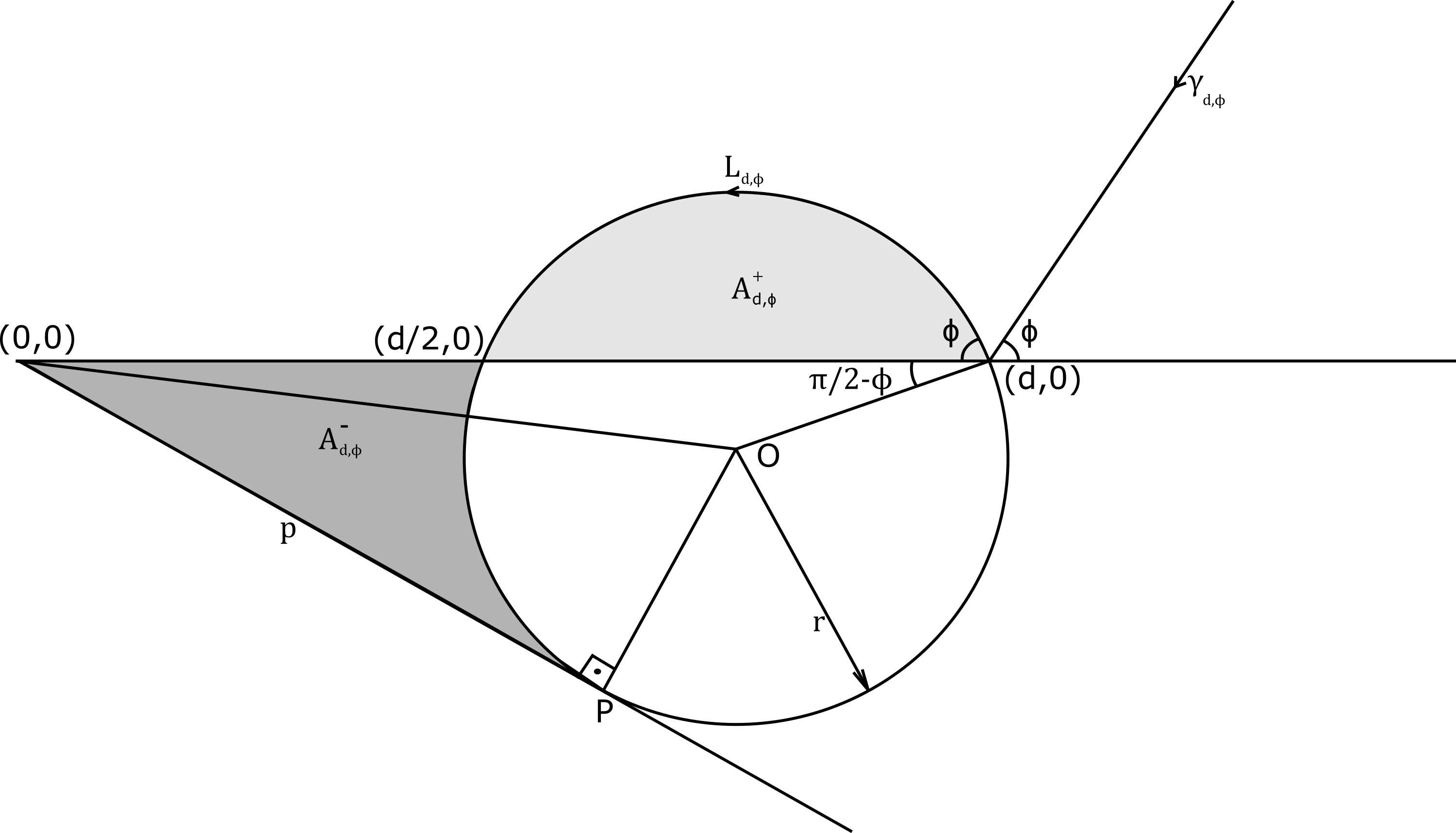}
\caption{The situation described in Subsection \ref{sssec:critical}. A projection of a half-line that reflects at a point in $H$, and its continuation after the reflection point.}\label{fig:critical}
\end{figure}

Consider first a sequence of horizontal curves $\gamma_{d,\phi}: \RR^- \to \RR^3$ lifting a half-line and hitting $H$ at the point $(d,0,0)$, where they make an angle $\phi$ with the radial half-line. Their magnetic momentum is $a^- =0$ which, after the reflection, increases to $a^+ = 4\frac{\sin(\phi)}{d}$. It follows that the trajectory $\gamma_{d,\phi}$ continues after the reflection as the lift of a circular segment with radius $r := \frac{d}{4\sin(\phi)}$. This subsequent segment $L_{d,\phi}$ may hit $H$ again after finite time. Our goal now is to determine whether this is the case. We remark that this will depend only on $\phi$, but not on $d$: the transformation $(x,y,z) \to (\lambda x,\lambda y, \lambda^2 z)$ is conformal and preserves the table and the billiard trajectories (up to reparametrisation). In particular, $L_{d,\phi}$ is identified with $L_{\lambda d,\phi}$, showing that either both of them intersect again or $H$, or none do. 

We can readily compute that the center of $\pi \circ L_{d,\phi}$ is the point 
\[O := \left(\frac{3d}{4},-\cotan(\phi)\frac d 4\right) =(3\sin(\phi)r,-\cos(\phi)r),\] which is at distance $o := ||O|| = r\sqrt{1+8\sin^2(\phi)}$ from the origin. This implies that the other intersection of $L_{d,\phi}$ with the radial line through $(d,0)$ is exactly in $(d/2,0)$. Do note that the area bounded by $\pi \circ L_{d,\phi}$ between these two points 
\[ A^+_{d,\phi} := (\phi - \sin(\phi)\cos(\phi))r^2  = \left(\dfrac{\phi}{\sin^2(\phi)} -\cotan(\phi)\right)\dfrac{d^2}{16} \]
is positive. This implies that any subsequent intersection of $L_{d,\phi}$ with $H$ lies beyond $(d/2,0)$.

The $z$ coordinate along $L_{d,\phi}$ is decreasing in-between $(d/2,0)$ and $P$, the first tangency between $\pi\circ L_{d,\phi}$ and a line through the origin. I.e., any potential intersection with $H$ must take place between $(d/2,0)$ and $P$. One can compute the radius of $P$ by the power of the point $(0,0)$ with respect to the circle $\pi \circ L_{d,\phi}$, which is $d^2/2$ as we see from the $x-$axis:
\[p := ||P|| = \sqrt{d^2/2} = d/\sqrt 2 = 2\sqrt{2}\sin(\phi)r ,\]
i.e., independent of the angle of incidence. Consider now the area $A^-_{d,\phi}$ swept by $\pi \circ L_{d,\phi}$ between $(d/2,0)$ and $P$. It is bounded from above by the area swept by $\pi \circ L_{d,\phi}$ between the two tangents passing through $(0,0)$. Denoting $\psi=\measuredangle((0,0),O,P)$ we get 
\begin{align*}
    A^-_{d,\phi}\leq p r - \psi r^2 & = \left(2\sqrt{2}\sin(\phi)  - \arctan\left(2\sqrt2\sin(\phi)\right) \right)r^2;
\end{align*}
note that $\arctan$ takes values in $[0,\pi/2]$. We conclude:

\begin{proposition} \label{prop:critical}
For any $\phi$ sufficiently close to $\pi$, $L_{d,\phi}: \RR^+ \to \RR^3$ does not intersect $H$ again. That is, it is a half-helix projecting down to a circle of center $O$ and radius $r$. We then write $\gamma_{d,\phi}: \RR \to \RR^3$ for the entire billiard trajectory.

Furthermore: Fix $\phi$ sufficiently close to $\pi$. Then, as we take $d$ to zero, the sequence $L_{d,\phi}$ converges in $C^0_{loc}$ to the constant curve at the origin. I.e. the billiard trajectories $\gamma_{d,\phi}$ converge in $C^0_{loc}$ to a curve $\gamma$ that is a half-line through the origin over $\RR^-$ and the constant curve over $\RR^+$.

Numerical computation shows that ``sufficiently close to $\pi$'' includes the interval $\phi \in [1.584 \approx 0.5042\pi,\pi)$.
\end{proposition}
\begin{proof}
For the first claim we observe that the condition of non-intersection of $L_{d,\phi}$ is equivalent to $A^+_{d,\phi} \geq A^-_{d,\phi}$. Due to our upper bound for $A^-_{d,\phi}$, it is sufficient that 
\[ (\phi - \sin(\phi)\cos(\phi))r^2 \quad\geq\quad  \left(2\sqrt{2}\sin(\phi)  - \arctan\left(2\sqrt2\sin(\phi)\right) \right)r^2, \]
equivalently:
\[ \left[\phi + \arctan\left(2\sqrt2\sin(\phi)\right)\right] - \left[\sin(\phi)\cos(\phi) + 2\sqrt{2}\sin(\phi)\right] \quad \geq 0. \]
The first bracket can be seen to be increasing, whereas the second one is decreasing in the interval $[\pi/2,\pi]$, making the entire expression positively monotone in $[\pi/2,\pi]$. Then, numerical inspection shows that there is a zero in $\psi \cong 1.584$.

The other two claims are then immediate. Our formulas state that, once $\phi$ is frozen, the values $r$ and $o$ go to zero as $d \to 0$. That is: the $L_{d,\phi}$ are lifts, defined over $\RR^+$, of circular trajectories whose radius and center go to zero. Remark \ref{rem:convergingToConstantExample} says that such a sequence must converge to the constant curve. On the other hand, the rest of $\gamma_{d,\phi}$, which is defined over $\RR^-$, has the usual parametrisation as a lift of a half-line. This is preserved in the limit, allowing us to conclude the proof.
\end{proof}

One can show that the opposite statement is true: as $\phi \to 0$, intersections with $H$ will appear. It would be interesting to understand this behavior and thus determine the subsequent reflections in terms of $\phi$. One could go further and attempt to analyse the situation where the ingoing radius of curvature is non-zero.

\begin{question}
Understanding the boundary foliation plays a central role in $3$-dimensional Contact Topology. It is known that its dynamics (not just singular points, but also periodic orbits) encode meaningful topological information of the ambient contact structure \cite{Giroux,Geiges}. An interesting open question is whether the type of a critical point can be read from the billiard reflections around it.\hfill$\triangle$
\end{question}

\begin{question}
The boundary foliation is now being studied from a sub-Riemannian perspective as well \cite{BBCH,BBC}, in an attempt to understand its metric properties (not unlike what we propose here). A particularly intriguing situation happens at elliptic points: Any two points in a neighbourhood can be connected by a boundary geodesic that is piecewise (one first goes along the distribution to the elliptic point, and then switches to another leaf of the distribution). Is it possible to produce such a piecewise geodesic as a gliding orbit? \hfill$\triangle$
\end{question}


\subsection{A band between two planes}\label{ssec:horizontalBand}

Now we concentrate on the case in which the boundary of the table consists of two parallel planes, i.e. $U=\RR^2\times [0,H]$, where $H>0$. We can show:
\begin{lemma} \label{lem:horizontalBand}
$2$-bounce orbits in the band can be obtained as follows: Fix a positive integer $n$ and a radius $r_0=\sqrt{\frac{H}{\pi n}}$. Then, there is a unique $2$-bounce trajectory $\gamma$ satisfying:
\begin{itemize}
    \item Its reflection points are $(r_0,\theta_0,0)$ and $(r_0,\theta_0,H)$.
    \item Both segments inbetween reflections are reversals of one another.
    \item Both segments lift a curve consisting of $n$ loops along the circle of radius $\rho=\frac{r_0}2$ and center $(r_0/2,\theta)$.
\end{itemize}
\end{lemma}
\begin{proof}
According to the analysis in Corollary \ref{cor:selfReflectionHorizontalPlane}, there is a tilted helix segment $\gamma$ reflecting to its reversal at $(r_0,\theta_0,0)$ and lifting a circular arc of radius $\rho=\frac{r_0}2$ and center $(r_0/2,\theta)$. This curve passes over $\pi \circ \gamma(t) = (0,0)$ if and only if $t \in \pi r_0(2\NN+1)$. Similarly, it satisfies $\pi\circ\gamma(t)=(r_0,\theta_0)$ if and only if $t\in2\pi r_0\NN$. This allows us to solve $r_0$ in terms of $H$, as claimed. The number $n$ corresponds to the number of loops in the projection in-between reflection points.
\end{proof}

\begin{remark}
It is unclear to the authors whether other 2-bounce orbits exist in this setting. If they do, the two smooth segments defining them cannot be reflections of one another. \hfill$\triangle$
\end{remark}

\subsection{The finite cylinder}\label{ssec:finiteCylinder}

Lastly, we consider a finite cylinder, combining the situations from Subsections \ref{ssec:standardCylinder} and \ref{ssec:horizontalBand}. Do note that this is a table with piecewise boundary. We write $D_R(0,0)\times [0,H]$ for the cylinder of height $H$ and base a disc of radius $R$ centered at the origin. We restrict ourselves to the search for periodic orbits.

Let us make the preliminary observation that some of the periodic orbits found in Lemmas \ref{lem:ngon} and \ref{lem:horizontalBand} are still there, for appropriate values of $R$ and $H$. Indeed, for the former we must check that the total height variation is small enough to avoid reflections of the trajectories with the top and bottom boundaries. For the second Lemma to apply, and using the same notation as there, we must consider $r_0 < R$. We leave the details to the reader.

\subsubsection{A tilted cylinder}

More interestingly, we want to produce examples of periodic trajectories reflecting at both boundaries. At first sight, it seems hard to combine the two constructions, since the trajectories found in Lemma \ref{lem:horizontalBand} pass through the origin and reflect at their furthest point from it, whereas those in Lemma \ref{lem:ngon} consist of circular segments that do not enclose the origin. We circumvent this by choosing a different table. Namely, we work on $D_R(d,0)\times [0,H]$, a cylinder centered now at $(d,0,0)$ with $d > R$. We could, equivalently by the skew symmetry~(\ref{shearmap}), choose a cylinder centered at the $z$-axis but with bottom and top lids of slope $d$.

To construct our periodic orbit $\gamma$ in $D_R(d,0)\times [0,H]$, we begin with a magnetic orbit $\nu$ in the base $D_R(d,0)$. It consists of several circular segments, all equal to one another by rotations around $(d,0)$. We denote one of them by $\beta$; we choose its radius $r$ to satisfy $2r \in (d-R,d+R)$ and its centre to be $(r,0)$. We write $q_0, q_1$ for the two intersection points between $\beta$ and the boundary circle $\partial D_R(d,0)$. We also denote $2\varphi := \measuredangle(q_0,(d,0),q_1)$. Then, $\varphi: (d-R,d+R) \to (0,\pi)$, sending the chosen $r$ to the angle, is an increasing homeomorphism.

Let us assume $2\varphi= 2\pi\frac{a}{b}$, for some rational $a/b$ with $a$ and $b$ coprime. Then, just like in Lemma \ref{lem:ngon}, $\beta$ can be continued to a periodic magnetic billiard orbit with $b$ reflections. We let $\nu$ be the $c$-fold cover of such a trajectory, with $c$ some positive integer. We choose $\nu(0) = (2r,0)$; this starting point, once lifted, is meant to be the reflection point of $\gamma$ with the lower boundary. Similarly, letting $T$ be the arclength of $\nu$, we have that $\nu(T) = (2r,0)$ will lift to a reflection point of $\gamma$ with the upper boundary.

Write then $\gamma: [0,T]\to D_R(d,0)\times \RR$ for the lift with prescribed initial point $\gamma(0)=(2r,0,0)$. Assuming $\gamma(T)=(2r,0,H)$ and that no other intersections with the upper and lower boundaries exist, we will have, by Corollary \ref{cor:selfReflectionHorizontalPlane}, that $\gamma$ reflects to itself at both $0$ and $T$, producing a periodic orbit with $2bc$ reflections with the vertical boundary, one reflection with the top, and one with the bottom.

We then note that the first assumption $\gamma(T)=(2r,0,H)$ holds as long as the area bounded by $\nu$ is $H$. I.e., given $d$, $R$, $a/b$, and $c$, there is a single value $H$ with this property.

\begin{figure}[h]
\centering
\includegraphics[width=0.9\textwidth]{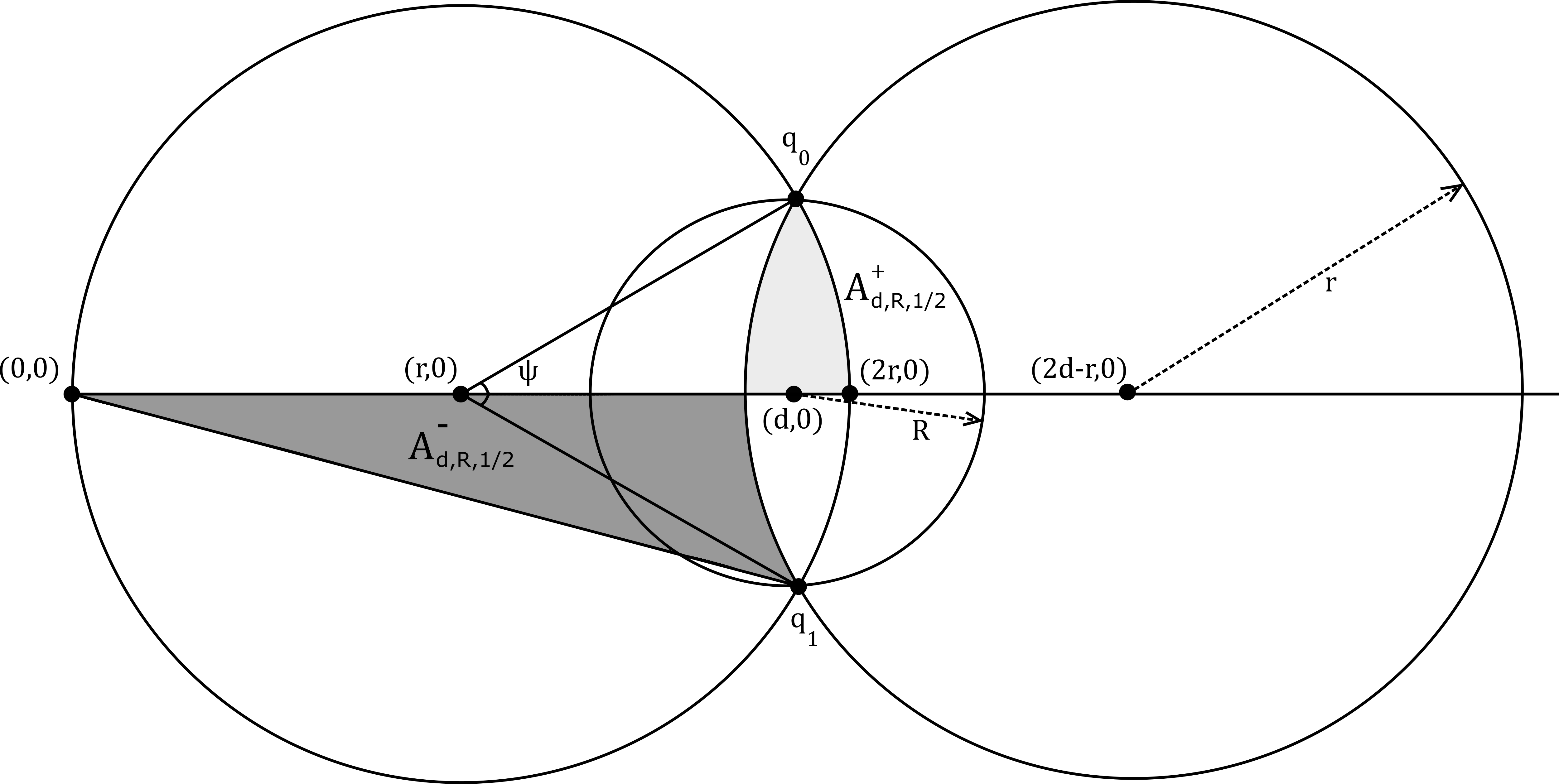}
\caption{The bigon from Lemma \ref{lem:finiteCylinder} under the assumption that $q_1$ appears in the arc before the tangency with respect to the origin. The lightly shaded area is positive and the dark area is negative. In this case, the latter is larger and therefore the trajectory hits the bottom boundary; that is, $\psi$ is not in the range described in the Lemma.}\label{fig:finiteCylinderA}
\end{figure}

\begin{figure}[h]
\centering
\includegraphics[width=0.9\textwidth]{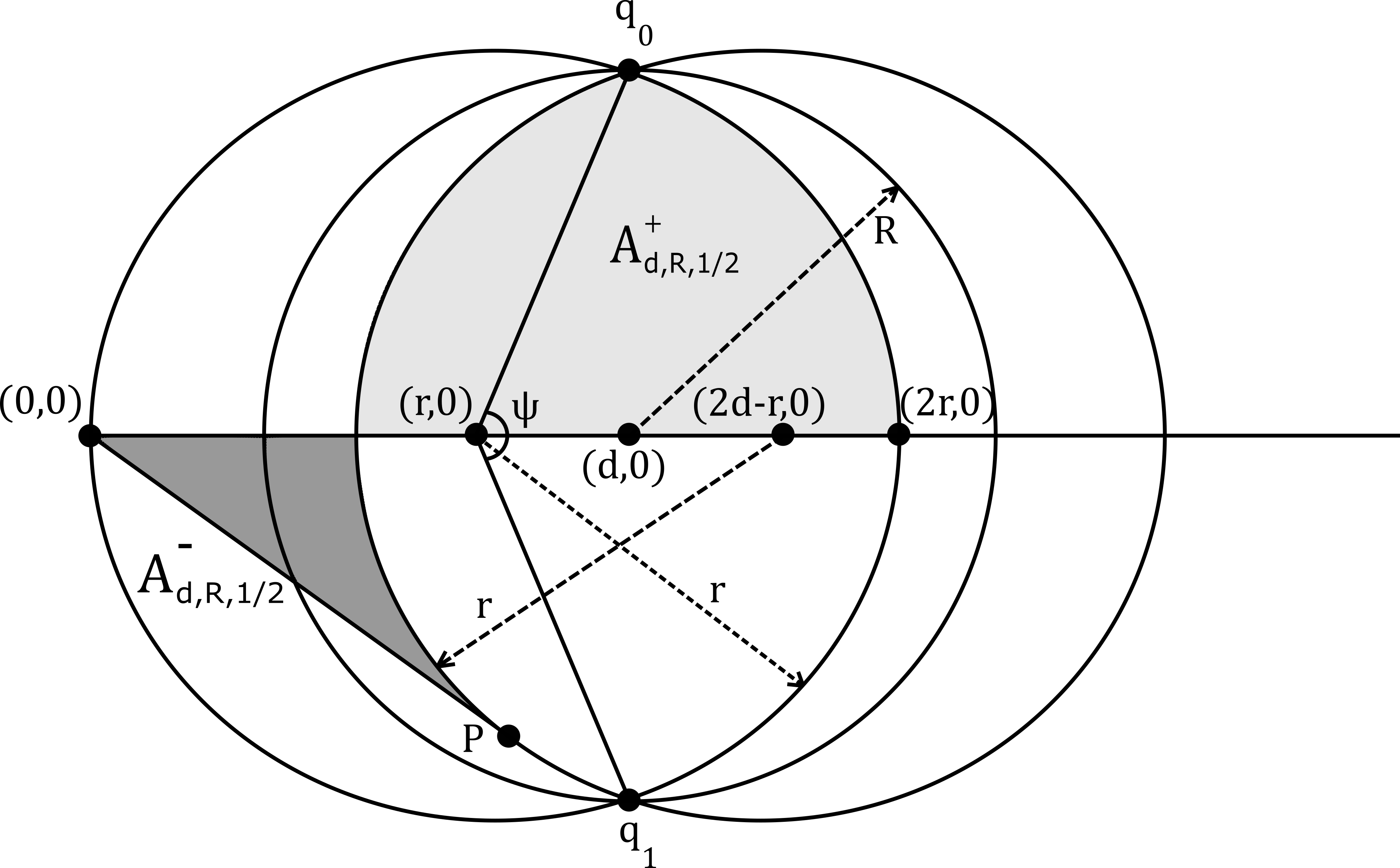}
\caption{The bigon from Lemma \ref{lem:finiteCylinder} under the assumption that $q_1$ lies beyond the tangency point from the origin. Here the positive area is larger than the negative area, so we obtain one of the trajectories described in the Lemma.}\label{fig:finiteCylinderB}
\end{figure}

We then explore whether the second assumption holds for $a/b=1/2$; we leave it to the reader to explore other values of $a/b$. The argument resembles that of Subsection \ref{sssec:critical} and amounts to controlling whether the $z$-coordinate of the lift remains in the interval $[0,H]$. Since $a/b=1/2$, the two points $q_i$ are $(d,\pm R)$ and $\gamma$ projects to a bigon with circular sides of radius $r$ and centers $C_0=(r,0)$ and $C_1=(2d-r,0)$. Clearly $2r>d$ and, since $(d-r)^2+R^2=r^2$, the value $r(R) = \frac{R^2+d^2}{2d}$ is uniquely determined from $d$ and $R$.

We depict the situation in Figures~\ref{fig:finiteCylinderA} and \ref{fig:finiteCylinderB} (we shall see shortly that there are two possible cases depending on the ratio $R/d$). The condition that the $z$-coordinate remains positive throughout the lift is equivalent to:
\begin{align*}
     A^+_{d,R,\frac12} \geq & A^-_{d,R,\frac12},
\end{align*}
where the $A^\pm_{d,R,\frac12}$ are as depicted in the figures.

In both situations we have that the positive area is given by the expression:
\[ A^+_{d,R,\frac12} = \area\left( D_r(C_0) \,\cap\, D_r(C_1) \,\cap\, \{y>0\}\right).\]
However, for the negative area $A^-_{d,R,\frac12}$ we need to distinguish two cases:
\begin{itemize}
    \item If the angle $\measuredangle( (0,0), q_1, (2d-r,0))$ is larger than $\pi/2$, as shown in Figure~\ref{fig:finiteCylinderA}, it holds that: 
\[ A^-_{d,R,\frac12} =\area\left(  D_r(C_1)^c \,\cap\, \{y<0\} \,\cap\, \{(x,y)\mbox{ lies above } \overline{(0,0)q_1}\}\right). \]
\item Otherwise, if $\measuredangle( (0,0), q_1, (2d-r,0)) $ is smaller than $\pi/2$ , as shown in Figure~\ref{fig:finiteCylinderB}), we define $P$ to be the tangency to the circle $S_r((2d/r,0))$ through $(0,0)$, and then:
\[ A^-_{d,R,\frac12} =\area\left(  D_r(C_1)^c \,\cap\, \{y<0\} \,\cap\, \{(x,y)\mbox{ lies above }  \overline{(0,0)P} \}\right). \]
\end{itemize}

Following the Figures, it is convenient to define the angle $2\psi=\measuredangle(q_0,(r,0),q_1)$. It reads:
\[ \psi(R) = \arccos\left(\dfrac{d^2-R^2}{d^2+R^2}\right) \quad\in\quad [0,\pi/2], \]
which is an increasing function of $R$. This implies that we can regard $(d,\psi)$ as the data defining the table, and solve for $R$. Then, the situation can be summarised as:
\begin{lemma} \label{lem:finiteCylinder}
Let $d$, $\psi$ and $c$ be given. Fix $a/b =1/2$ and carry out the previous construction, yielding constants $R$, $r$ and $H$, and a curve $\gamma_{d,R,c}$. Then, for $\psi$ sufficiently close to $\pi/2$, the curve $\gamma_{d,R,c}$ is a periodic orbit in $D_R(d,0)\times [0,H]$ with $2+2c$ bounces. 

Numerical computation shows that ``sufficiently close to $\pi/2$'' contains the interval
\[ \phi \in [0.909324 \approx 0.289447\pi,\pi/2). \]
\end{lemma}
\begin{proof}
We fix $d$ and vary $\psi$ (equivalently, we vary $R$). We observe that the formulas $(d-r)=r\cos(\psi)$ and $R=r\sin(\psi)$ hold.

We first compute the transition between both situations. This happens when $\measuredangle( (0,0), q_1, (2d-r,0)) = \pi/2$, which is equivalent to any of the conditions:
\begin{align*}
\overline{(0,0)q_1}^2+\overline{q_1(2d-r,0)}^2 &= \overline{(0,0)(2d-r,0)}^2, \\
R^2 + d^2 + r^2 &  = (2d-r)^2, \\
R^2 & = 3d^2- 4dr = 3(d-r)^2 + 2(d-r)r - r^2 \\
0 & = 4\cos(\psi)^2 + 2\cos(\psi) - 2.
\end{align*}
Using the last one we solve for $\psi$ and see that the transition happens at $\psi = \frac \pi3$. This is equivalent to $R = d/\sqrt 3$ and $r = 2d/3$. 

In the first case $\psi\leq\pi/3$, the difference of the two areas of interest reads:
\begin{align*}
 \triangle_{d,R} &:=A^+_{d,R,\frac12}-A^-_{d,R,\frac12} = \frac32\psi r^2 - \frac{3}{2}R(d-r) - \frac12 dR \\
 & = \frac32\psi r^2 - 2R(d-r) -\frac{1}{2}Rr \\
  &= \frac12 r^2(3\psi -4\sin(\psi)\cos(\psi) - \sin(\psi)),\\
\end{align*}
which is positive if and only if $3\psi -4\sin(\psi)\cos(\psi) - \sin(\psi)$ is positive. For this, note that for $\psi\in[\pi/4,\pi/3]$ we have $\sin(\psi)^2\geq\cos(\psi)^2$ and thus $\triangle_{d,R}'(\psi) = 4 \sin^2(\psi) - 4 \cos^2(\psi) - \cos(\psi) + 3>0$. The function $3\psi -4\sin(\psi)\cos(\psi) - \sin(\psi)$ has a zero that we numerically determined to be at $0.909324 \approx 0.289447\pi$. This proves the claim for $\phi\in [0.909324 \approx 0.289447\pi,\pi/3]$.

In the second case $\psi\geq\pi/3$, the difference between the two areas of interest reads:
\begin{align*}
 \triangle_{d,R} :=A^+_{d,R,\frac12}-A^-_{d,R,\frac12} & = \psi r^2 - R(d-r) -\frac12r\sqrt{2(d-r)2d} + \frac12\arcsin\left(\frac{r}{2d-r}\right)r^2\\
 &= r^2\left(\psi-\cos\psi\sin\psi-\sqrt{\cos\psi(\cos\psi+1)}+\frac 12\arctan\left(2\sqrt{\cos\psi(\cos\psi+1)}\right)\right).
\end{align*}
We then compute the derivative of $\triangle_{d,R}(\psi)/r^2$:
\[ 2 \sin^2(x) + \frac{2\sin(\psi)(2\cos(\psi)+1)\sqrt{\cos\psi(\cos\psi+1)}}{1+4\cos\psi(\cos\psi+1)}, \]
which is positive in the interval $\psi \in [\pi/3,\pi/2]$. Together with $\triangle_{d,R}(\pi/3) = \pi/2 - 3 \sqrt3/4>0$, this concludes the proof.
\end{proof}

\begin{remark}
If we take $R=d$, the cylinder of interest contains the vertical line through the origin. Then, as $R \to d$, the corresponding sequence of curves $\gamma_{d,R,c}$ converges to one of the curves described in Lemma \ref{lem:horizontalBand}, which is a boundary geodesic with two reflection points. I.e., the $2c$ bounces with the vertical component of the boundary smooth out in the limit, and the two reflections that survive take place in the corner of the table. \hfill$\triangle$
\end{remark}

\section{Appendix: Non-autonomous settings}\label{appendix:nonAutonomous}

In order to maximise generality, let us also discuss non-autonomous (i.e., time-dependent) settings. We give the relevant definitions and point out what changes need to be made.

\subsection{Time-dependent control systems}

The following generalises the autonomous Definition \ref{def:controlSystem}:
\begin{definition} 
Given a smooth manifold $M$ and a time interval $I$, a \textbf{non-autonomous control system} over $M \times I$ is:
\begin{itemize}
\item A locally trivial fibre bundle $C \to M \times I$ whose fibres are manifolds (possibly with boundary or corners).
\item A fibrewise map $\rho: C \to TM$, called the \textbf{anchor}, that is smooth in $M$ and essentially bounded in $I$.
\end{itemize}
\end{definition}
Controls are defined exactly in the same manner as in the standard setting. We remark that this definition was used in several proofs within Section \ref{sec:controlTheory}, for instance Lemma \ref{lem:infEndpointMap} and Propositions \ref{prop:PSP} and \ref{prop:PMP}. In particular, those statements, as well as Filippov's theorem, translate immediately to the time-dependent world.

\subsection{Time dependent sub-Finsler}

Now we adapt the definitions from Section~\ref{sec:sub-FinslerSystems} to allow the distribution $\xi_t$ and the Finsler norm $f_t$ to vary smoothly with time. This makes the sub-Finsler control system from Definition \ref{def:sub-FinslerControl} non-autonomous:
\[ \rho_t: C_t:=\{v \in \xi_t \mid f_t(x) \leq 1\} \quad\to\quad TM. \]
We make the additional assumptions that $\xi_t$ is at all times bracket generating and that, for some auxiliary autonomous Finsler norm $f_0$ on $TM$, we have for all $v\in \xi \neq 0$ that $\frac{f_t(v)}{f_0(v)}$ is bounded away from 0 uniformly in space and time.

As remarked previously, the Pontryagin maximum principle Proposition \ref{prop:PMP} still holds. However, the maximised Hamiltonian depends on time
\begin{align*}
H_{f_t}: T^*M \quad\to\quad& \RR \\
H_{f_t}(\lambda) \quad := \quad & \max_{v \in \xi_t, f_t(v) \leq 1} \lambda(v)
\end{align*}
and therefore, given a Hamiltonian curve $\dot\lambda = X_{H_{f_t}}$, the quantity $\dot H_{u_t}(\lambda(t))$ may not be zero. This implies that Corollary~\ref{Cor:preservedQuantity} does not hold in general. This makes it harder to use the sub-Finsler cogeodesic flow from Definition~\ref{def:cogeodesicFlow}. In the following, we work with a varying energy level. This leads us to strengthen the assumption of normality:
\begin{assumption}
We restrict our attention to those Hamiltonian trajectories of the maximised Hamiltonian that are normal, in the sense that $H_{f_t}(\lambda(t))\neq 0 \quad \forall t$. \hfill$\triangle$
\end{assumption}
It is important to remark that one could weaken our assumption that each $\xi_t$ is bracket-generating and, assuming smoothness in $t$, instead ask for $\xi_t \oplus \langle \partial_t \rangle$ to be a bracket-generating distribution in $M \times I$. We will not explore this in the current article.

\subsection{Non-autonomous diamond control systems}

In order to define a billiard reflection, our original approach of the diamond control system needs to be modified. We choose to do this by adding drift in an auxiliary two dimensional time.

\begin{definition}
We define the \textbf{non-autonomous diamond control system} associated to $\rho: C_t \to TM$ to be 
\[ \rho_{M \times M}: C_{M \times M} \to TM \times TM\times T\RR\times T\RR \]
where
\begin{itemize}
\item $C_{M \times M}(q_1,q_2,t_1,t_2) := C_{t_1} \times C_{t_2} \times [0,1]$,
\item $\rho_{M \times M}(u_1,u_2,t_1,t_2,s) := (s\rho_{t_1}(u_1),-(1-s)\rho_{t_2}(u_2),s,-(1-s))$.
\end{itemize}
Its image can be expressed as follows:
\[ \{(s v, -(1-s)w) \in TM \times TM \mid s \in [0,1], v,w \in \rho(C) \}. \]
\end{definition}

Consider $\gamma: [0,t_1+t_2] \to M$, an admissible curve from $\gamma(0) = q_1$ to $\gamma(t_1+t_2) = q_2$ with $\gamma(t_1) \in \partial M$ a point in the boundary. Passing to the diamond control system, we can define a curve $\gamma_{M \times M}: [0,t_1+t_2] \to M \times M\times\RR\times\RR$:
\[ \gamma_{M \times M}(t) := \left(\gamma|_{[0,t_1]}\left(t\dfrac{t_1}{t_1+t_2}\right),\gamma|_{[t_1,t_1+t_2]}\left((t_1+t_2-t)\dfrac{t_2}{t_1+t_2}\right),t\dfrac{t_1}{t_1+t_2},(t_1+t_2-t)\dfrac{t_2}{t_1+t_2} \right)\]
connecting the line $\{(q_1,q_2,0)\}\times\RR$ with the contemporary diagonal of the boundary $\Delta_{\partial M}\times\Delta_{\RR^2} := \{ (q,q,t,t)\mid q\in\partial M, t\in\RR \}$. By construction, $\gamma_{M \times M}$ is admissible for the diamond system. Lemma~\ref{rectangle} about the families $\Fcal_\gamma$ and Lemma~\ref{lem:diamondMinimizer} on the minimizing conditions still hold true in this setup. Note that, assuming $\dot\gamma\in\partial \rho(C)$, the speed of the two respective times and spaces are coupled: 
\[f_{\gamma_3}(\dot\gamma_1)=\|\partial\gamma_3\|,\]
\[f_{\gamma_4}(\dot\gamma_2)=\|\partial\gamma_4\|.\]
This means that the length of the total curve $\gamma_{M\times M}$ coincides with the position in the starting line $\gamma_4(0)=length(\gamma_{M\times M})$.

\subsubsection{Reflection law}
According to Proposition~\ref{prop:PMPboundary}, the starting momentum $\lambda_{M\times M}$ of $\gamma_{M\times M}$ annihilates $T\{(q_1,q_2,0)\}\times\RR$. This is in harmony with our assumption that $\gamma_{M\times M}$ optimizes not only the length from $(q_1,q_2,0,t_1+t_2)$ to $\Delta_{\partial M}\times\Delta_{\RR^2}$, but also the total time taken by the trip. We deduce:
\begin{proposition}\label{prop:NONAUTreflectionLaw}
Let $\gamma: [0,t_1+t_2] \to U$ be an admissible curve of the non-autonomous sub-Finsler setting, connecting $\gamma(0) = q_1$ to $\gamma(t_1+t_2) = q_2$, and with $\gamma(t_1) \in \partial U$ the only point in the boundary.

Then, $\gamma$ is a minimiser of length among such curves if and only if there is a cotangent lift $\lambda: [0,t_1+t_2] \to T^*U$ such that the pieces $\lambda|_{[0,t_1]}$ and $\lambda|_{[t_1,t_1+t_2]}$ satisfy PMP and such that
\begin{equation} \label{eq:nonAutonomousReflectionLaw}
\lambda(t_1^-) - \lambda(t_1^+) \in \Ann(\partial U),
\end{equation}
where $t_1^-$ denotes the limit as we approach $t_1$ from the left and $t_1^+$ is the limit from the right.
\end{proposition}

As mentioned in the original argument, Equation~\ref{eq:nonAutonomousReflectionLaw} is not enough to conclude a reflection law. We have to additionally assume that a reflection at time $t$ preserves $H_{f_t}$. That is, from a symplectic perspective, whenever a reflection is going to happen at time $t$, we consider the microlocal realisation of the current energy level of $H_{f_t}$, and perform the usual reflection. This is consistent with the idea that reflections are instantaneous, and therefore they only depend on the current time. One can then perform the usual analysis of degeneracies of the reflection by considering the interaction between $\xi_t$ and the boundary of the table.


\end{document}